\documentclass{aptpub2}

\authornames{S. Grusea, S. Mercier} 
\shorttitle{Distribution for the maximal segmental score of a Markov chain} 

\usepackage{graphicx}
\begin{document}

\def \P{\mathbb P}
\def \R{\mathbb R}
\def \N{\mathbb N}
\def \Z{\mathbb Z}
\def \E{\mathbb E}
\def \X{\mathbb X}
\def \Y{\mathbb Y}
\def \A{\mathbb A}
\def \B{\mathbb B}
\def \V{\mathbb V}
\def \W{\mathbb W}

\def \Seq{\mathbb S}

\def \leqp{\leqslant}
\def \geqp{\geqslant}
\newcommand{\1}{{1\hspace{-0.2ex}\rule{0.12ex}{1.61ex}\hspace{0.5ex}}}

\title{Improvements on the distribution of maximal segmental scores in a Markovian sequence} 

\authorone[Institut de Math\'{e}matiques de Toulouse, Universit\'e de Toulouse, INSA de Toulouse]{S. Grusea}
\addressone{Institut National des Sciences Appliqu\'{e}es, 135 avenue de Rangueil, 31400, Toulouse, France}
\authortwo[Institut de Math\'{e}matiques de Toulouse, Universit\'e de Toulouse, Jean Jaur\`es]{S. Mercier} 
\addresstwo{Institut de Math\'{e}matiques de Toulouse, UMR5219, 
Universit\'{e} de Toulouse 2 Jean Jaur\`{e}s, 5 all\'{e}es Antonio Machado, 31058, Toulouse, Cedex 09, France} 

\begin{abstract}
Let $(A_i)_{i \geq 0}$ be a finite state irreducible aperiodic Markov chain and $f$ a lattice score function such that the average score is negative and positive scores are possible. Define $S_0:=0$ and $S_k:=\sum_{i=1}^k f(A_i)$ the successive partial sums, $S^+$ the maximal non-negative partial sum, $Q_1$ the maximal segmental score of the first non-negative excursion and $M_n:=\max_{0\leq k\leq\ell\leq n} (S_{\ell}-S_k)$ the \textit{local score} first defined by Karlin and Altschul \cite{KAl90}. We establish recursive formulae for the exact distribution of $S^+$ and derive new approximations for the distributions of $Q_1$ and $M_n$.  Computational methods are presented in a simple application case and comparison is performed between these new approximations and the ones proposed in \cite{KDe92} in order to evaluate improvements. 
\end{abstract}



\section{Introduction}\label{Sec:Introduction}

There is nowadays a huge amount of biological sequences available. The \textit{local score} for one sequence analysis, first defined by Karlin and Altchul in \cite{KAl90} (see Equation (\ref{Def:LocalScore}) below for definition) quantifies the highest level of a certain quantity of interest, e.g. hydrophobicity, polarity, etc..., that can be found locally inside a given sequence. This allows for example to detect atypical segments in biological sequences. In order to distinguish significantly interesting segments from the ones that could have appeared by chance alone, it is necessary to evaluate the $p$-value of a given local score. Different results have already been established using different probabilistic models for sequences: independent and identically distributed variables model (i.i.d.) \cite{MCC03,KAl90,KDe92,MDa01}, Markovian models \cite{KDe92, HMe07} or Hidden Markov Models \cite{DKM98}. In this article we will focus on the Markovian model.\par
An exact method was proposed by Hassenforder and Mercier \cite{HMe07} to calculate the distribution of the local score for a Markovian sequence, but this result is computationally time consuming for long sequences ($>10^3$).
Karlin and Dembo \cite{KDe92} established the limit distribution of the local score for a Markovian sequence and a random scoring scheme depending on the pairs of consecutive states in the sequence.
They proved that the distribution of the local score is asymptotically a Gumble distribution, as in the i.i.d. case. In spite of its importance, their result in the Markovian case is unforfunately very little cited or used in the literature. A possible explanation could be the fact that the random scoring scheme defined in \cite{KDe92} is more general than the ones classically used in practical approaches. In \cite{GRH06} and \cite{FBM17}, the authors verify by simulations that the local score in a certain dependence model follows a Gumble distribution, and use simulations to estimate the two parameters of this distribution. 

In this article we study the Markovian case for a more classical scoring scheme. We propose a new approximation for the distribution of the local score of a Markovian sequence. We compare it to the one derived from the result of Karlin and Dembo \cite{KDe92} and illustrate the obtained improvement in a simple application case. 

\subsubsection*{Mathematical framework}

Let  $(A_i)_{i \geq 0}$ be an irreducible and aperiodic Markov chain taking its values in a finite set ${\cal A}$ containing $r$ states denoted $\alpha$, $\beta$, $\dots$ for simplicity. Let ${\bf P}=(p_{\alpha\beta})_{\alpha,\beta}$ be its transition probability matrix and $(\pi_{\alpha})_{\alpha}$ its stationary frequency vector. In order to simplify the presentation, we suppose that ${\bf P}$ is positive ($\forall \alpha,\beta,\ p_{\alpha\beta}>0$).
We also suppose that the Markov chain is stationary, i.e. with initial distribution of $A_0$ given by $\pi$.
$\P_{\alpha}$ will stand for the conditional probability given $\{A_0 = \alpha\}$.
We consider a lattice score function $f: {\cal A}\rightarrow d\Z$, with $d\in\N$ being the lattice step. Note that, since ${\cal A}$ is finite, we have a finite number of possible scores.
Since the Markov chain $(A_i)_{i \geq 0}$ is supposed to be stationary, the distribution of $A_i$ is $\pi$ for every $i \geq 0$. We will simply denote $\E[f(A)]$ the average score. 

In this article we make the hypothesis that the average score is negative, i.e.
\begin{eqnarray}\label{Hyp:ScoreMoyNeg}
\E[f(A)]=\sum_{\alpha}f(\alpha)\pi_{\alpha}<0.
\end{eqnarray}
We will also suppose that for every $\alpha \in \mathcal{A}$ we have
\begin{eqnarray}\label{Hyp:ProbaScorePos}
\P_{\alpha}(f(A) > 0) > 0 \mbox{ and }\ \P_{\alpha}(f(A) < 0) > 0.
\end{eqnarray}

Let us introduce some definitions and notation. Let $S_{0}:=0$ and denote 
$$S_{k}:=\sum_{i=1}^kf(A_i),$$ 
for $k\geq 1$ the successive partial sums. Let $S^+$ be the \textit{maximal non-negative partial sum} 
$$S^+:=\max\{0,S_k : k \geq 0\}.$$ 

Further, let $\sigma^-:=\inf\{k\geqp 1:S_k<0\}$ be the time of the first negative partial sum. Note that $\sigma^-$ is an a.s.-finite stopping time due to (\ref{Hyp:ScoreMoyNeg}).

Let $K_0 := 0$. For $i \geq 1$, we denote $K_i := \inf\{k > K_{i-1} : S_{k} - S_{K_{i-1}} < 0\}$ 
the successive decreasing ladder times of $(S_k)_{k\geq 0}$. Note that $K_1 = \sigma^-$.

Let us now consider the subsequence $(A_i)_{ 0 \leq i \leq n}$ for a given length $n \in \mathbb{N}\setminus \{0\}$.
Denote $m(n): = \max\{i \geq 0 : K_i \leq n\}$ the random variable corresponding to the number of decreasing ladder times arrived before $n$. For every $i=1,\dots,m(n)$, we call the sequence $(A_j)_{K_{i-1}< j\leq K_{i}}$ the $i$-th non-negative excursion. 

Note that, due to the negative drift, we have $\E[K_1] < \infty$ (see Lemma \ref{lem:Esp_K1}) and $m(n) \to \infty$ \textit{a.s.} when $n \to \infty$. 
To every non-negative excursions $i=1,\dots,m(n)$ we associate a \textit{maximal segmental score} (called also \textit{height}) $Q_{i}$ defined by
$$Q_i : = \max_{K_{i-1}\leq k < K_{i}} (S_k - S_{K_{i-1}}).$$

First introduced by Karlin and Altschul in \cite{KAl90}, the \textit{local score}, denoted $M_n$, is defined as the maximum segmental score for a sequence of length $n$:
\begin{equation}\label{Def:LocalScore}
M_{n}:=\max_{0\leqp k\leqp \ell\leqp n}(S_{\ell}-S_{k}).
\end{equation}
Note that $M_{n}=\max(Q_1,\dots,Q_{m(n)},Q^*)$,
with $Q^*$ being the maximal segmental score of the last incomplete non-negative excursion $(A_j)_{K_{m(n)}< j\leq n}$.
Mercier and Daudin \cite{MDa01} give an alternative expression for $M_n$ using the Lindley process $(W_k)_{k \geq 0}$ describing the excursions above zero between the successive stopping times $(K_i)_{i\geq 0}$. With $W_0:=0$ and $W_{k+1}:=\max(W_{k}+f(A_{k+1}),0)$, we have
$M_{n}=\max_{0\leqp k\leqp n}W_k.$

\begin{rem}\label{KD}
Karlin and Dembo \cite{KDe92} consider a random score function defined on pairs of consecutive states of the Markov chain: they associate to each transition $(A_{i-1},A_{i})=(\alpha,\beta)$ a bounded random score $X_{\alpha\beta}$ whose distribution depends on the pair $(\alpha,\beta)$. Moreover, they suppose that, for $(A_{i-1},A_i)=(A_{j-1},A_j)=(\alpha,\beta)$, the random scores $X_{A_{i-1} A_i}$ and $X_{A_{j-1} A_j}$ are independent and identically distributed as $X_{\alpha\beta}$. The framework of this article corresponds to the case when the score function is deterministic, with $X_{A_{i-1}A_{i}}=f(A_i)$.\par
Note also that in our case the hypotheses (\ref{Hyp:ScoreMoyNeg}) and (\ref{Hyp:ProbaScorePos}) assure the so-called cycle positivity, i.e. the existence of some state $\alpha \in \mathcal{A}$
satisfying $$\P\left(\bigcap_{k=1}^{m-1} \{S_k>0\} \left |\ A_0=A_m=\alpha\right. \right) > 0.$$
In \cite{KDe92}, in order to simplify the presentation, the authors strengthen the assumption of cycle positivity by assuming that
$\P(X_{\alpha\beta} > 0) > 0 \text{ and } \P(X_{\alpha\beta} < 0) > 0$ for all $\alpha, \beta \in \mathcal{A}$ (see (1.19) of \cite{KDe92}), but precise that the cycle positivity is sufficient for their results to hold. Note that hypotheses (\ref{Hyp:ScoreMoyNeg}) and (\ref{Hyp:ProbaScorePos}) are usually verified in biological applications.
\end{rem}

In Section \ref{Sec:MainResults} we first introduce few more definitions and notation. Then we present the main results: a recursive result for the exact distribution of the maximal non-negative partial sum $S^+$ for an infinite sequence in Theorem \ref{res:exactS+}; based on the exact distribution of $S^+$, we further propose new approximations for the distribution of the height of the first non-negative excursion $Q_1$ in Theorem \ref{res:Q1} and for the  distribution of the local score $M_n$ for a sequence of length $n$ in Theorem \ref{res:Mn}.  Section \ref{Sec:Proofs} contains the proofs of the results of Section \ref{Sec:MainResults} and of some useful lemmas which use techniques of Markov renewal theory and large deviations. In Section \ref{sec:comp} we propose a computational method for deriving the quantities appearing in the main results. A simple scoring scheme is developed in Subsection \ref{subsec:simplecase}, for which we compare our approximations to the ones proposed by Karlin and Dembo \cite{KDe92} in the Markovian case.

\section{Statement of the main results}\label{Sec:MainResults}

\subsection{Definitions and notation}\label{subsec:notdef}

For every $\alpha, \beta \in \mathcal{A}$, we denote $q_{\alpha\beta}:=\P_{\alpha}(A_{K_1} = \beta)$ and
${\bf Q}:=(q_{\alpha\beta})_{\alpha,\beta}$.
Define $\mathcal{A}^-=\{\alpha\in\mathcal{A}:f(\alpha)<0\}$ and $\mathcal{A}^+=\{\alpha\in\mathcal{A}:f(\alpha)>0\}$. Note that the matrix ${\bf Q}$ is stochastic, with $q_{\alpha\beta}=0$ for $\beta \in \mathcal{A}\setminus \mathcal{A}^-$. Its restriction ${\bf \tilde Q}$ to $\mathcal{A}^-$ is stochastic and irreducible. The states $(A_{K_i})_{i \geq 1}$ of the Markov chain at the end of the successive non-negative excursions define a Markov chain on $\mathcal{A}^-$ with transition probability matrix ${\bf \tilde Q}$. 
For every $i \geq 2$ we thus have
$\P(A_{K_i} = \beta \ | A_{K_{i-1}} = \alpha ) = q_{\alpha \beta}$
if $\alpha, \beta \in \mathcal{A}^-$ and 0 otherwise.
Denote $\tilde{z} > 0$ the stationary frequency vector of the irreducible stochastic matrix ${\bf \tilde Q}$ and let $z:=(z_{\alpha})_{\alpha\in\mathcal{A}}$ with $z_{\alpha}=\tilde{z}_{\alpha} > 0$ for $\alpha\in\mathcal{A}^-$ and $z_{\alpha}=0$ for $\alpha \in \mathcal{A} \setminus \mathcal{A}^-$. Note that $z$ is invariant for the matrix $Q$ i.e. $z{\bf Q}=z$.

\begin{rem}
Note that in Karlin and Dembo's Markovian model of \cite{KDe92} the matrix $Q$ is irreducible, thanks to their random scoring function and to their hypotheses recalled in Remark \ref{KD}.
\end{rem}

Using the strong Markov property, conditionally on $(A_{K_i})_{i \geq 1}$ the r.v. $(Q_i)_{i \geq 1}$ are independent, with the distribution of $Q_i$ depending only on $A_{K_{i-1}}$ and $A_{K_{i}}$.

For every $\alpha \in \mathcal{A}$, $\beta \in \mathcal{A}^-$ and $y \geq 0$, let 
$$F_{\alpha \beta}(y):=\P_\alpha(Q_1 \leq y \ | A_{\sigma^-} = \beta)\ 
\mbox{ and }
\ F_\alpha(y): = \P_\alpha(Q_1 \leq y).$$
Note that for any $\alpha \in \mathcal{A}^-$ and $i \geq 1$, $F_{\alpha \beta}$ represents the cumulative distribution function (\textit{cdf}) of the height $Q_i$ of the $i$-th non-negative excursion given that it starts in state $\alpha$ and ends in state $\beta$, i.e.
$F_{\alpha \beta}(y)=\P(Q_i \leq y \ | A_{K_i} = \beta, A_{K_{i-1}} = \alpha)$, whereas $F_\alpha$ represents the \textit{cdf} of $Q_i$ conditionally on the $i$-th non-negative excursion starting in state $\alpha$, i.e. $F_{\alpha}(y)=\P(Q_i \leq y \ | A_{K_{i-1}} = \alpha)$.

We thus have
$$F_\alpha(y) = \sum_{\beta \in \mathcal{A}} F_{\alpha \beta}(y) q_{\alpha \beta}=\sum_{\beta \in \mathcal{A}^-} F_{\alpha \beta}(y) q_{\alpha \beta}. $$

We also introduce the stopping time $\sigma^+:=\inf\{k\geqp 1:S_k>0\}$ with values in $\N \cup \{\infty\}$. Due to hypothesis $(\ref{Hyp:ScoreMoyNeg})$ we have $\P_\alpha(\sigma^+ < \infty) < 1$, for all $\alpha\in \mathcal{A}$.

For every $\alpha, \beta \in \mathcal{A}$ and $\xi > 0$, let
$$L_{\alpha\beta}(\xi):=\P_{\alpha}(S_{\sigma^+} \leqp \xi, \sigma^+<\infty, A_{\sigma^+}=\beta).$$
Note that $L_{\alpha\beta}(\xi) = 0$ for $\beta \in \mathcal{A} \setminus \mathcal{A}^+$.
We have $L_{\alpha\beta}(\infty) \leq \P_\alpha(\sigma^+ < \infty) < 1$, and hence
\begin{equation}\label{L}
\int_{0}^\infty dL_{\alpha\beta}(\xi) = 1 - L_{\alpha\beta}(\infty) > 0.
\end{equation}
Let us also denote
$$L_{\alpha}(\xi):=\sum_{\beta \in \mathcal{A}^+}L_{\alpha \beta}(\xi) = \P_{\alpha}(S_{\sigma^+} \leqp \xi, \sigma^+<\infty)$$
the conditional \textit{cdf} of the first positive partial sum when it exists, given that the Markov chain starts in state $\alpha$, and
$$L_{\alpha}(\infty) := \lim_{\xi \to \infty} L_{\alpha}(\xi) = \P_{\alpha}(\sigma^+<\infty).$$

For any $\theta \in \mathbb{R}$ we introduce the following matrix 
$$\Phi(\theta):=\left(p_{\alpha\beta}\cdot\exp(\theta f(\beta))\right)_{\alpha,\beta \in \mathcal{A}}.$$ 
Since the transition matrix $\bf{P}$ was supposed to be positive, by the Perron-Frobenius Theorem, the spectral radius $\rho(\theta) > 0$ of the matrix $\Phi(\theta)$ coincides with its do\-mi\-nant eigenvalue, for which there exists a unique positive right eigen vector $u(\theta)=(u_i(\theta))_{1\leq i \leq r}$ (seen as a column vector) normalized so that $\sum_{i=1}^r u_i(\theta)=1$. Moreover, $\theta \mapsto \rho(\theta)$ is differentiable and strictly log convex (see \cite{Lancaster,DKa91b,KOs87}). In Lemma \ref{lem:rho'} we prove that $\rho'(0) = \E[f(A)]$, hence $\rho'(0) < 0$ by Hypothesis $(\ref{Hyp:ScoreMoyNeg})$. Together with the strict log convexity of $\rho$ and the fact that $\rho(0)= 1$, 
this implies that there exists a unique $\theta^* > 0$ such that $\rho(\theta^*)=1$ (see \cite{DKa91b} for more details).
  \\


\subsection{Main results. Improvements on the distribution of the local score}

Let $\alpha \in \mathcal{A}$. We start by giving a result which allows to compute recursively the \textit{cdf} of the maximal non-negative partial sum $S^+$.
We denote by $F_{S^+,\alpha}$ the \textit{cdf} of $S^+$ conditionally on starting in state $\alpha$:
$$F_{S^+,\alpha}(\ell d) := \P_{\alpha}(S^+\leq\ell d), \ \ \forall \ell \in \N $$
and for every $k \in \mathbb{N}\setminus \{0\}$ and $\beta \in \mathcal{A}$:
$$L^{(k)}_{\alpha \beta}:=\P_\alpha(S_{\sigma^+} = k d, \sigma^+ < \infty,  A_{\sigma^+} = \beta).$$ 
Note that $L^{(k)}_{\alpha \beta}=0$ for $\beta \in \mathcal{A} \setminus \mathcal{A}^+$ and
$L_{\alpha}(\infty) = \sum_{\beta \in \mathcal{A}^+} \sum_{k=1}^{\infty} L^{(k)}_{\alpha \beta}.$

The following result gives a recurrence relation for the double sequence $(F_{S^+,\alpha}{(\ell d)})_{\alpha,\ell}$.
  
\begin{thm}[Exact result for the distribution of $S^+$]\label{res:exactS+}
For all $\alpha\in\mathcal{A}$ and $\ell \geq 1$:
\begin{align*}
F_{S^+,\alpha}(0) &=\P_\alpha(\sigma^+ = \infty) = 1 - L_{\alpha}(\infty),\\
F_{S^+,\alpha}{(\ell d)} &=1 - L_{\alpha}(\infty) + \sum_{\beta \in \mathcal{A}^+} \sum_{k=1}^{\ell} L^{(k)}_{\alpha \beta} \ F_{S^+,\beta}{((\ell-k)d)}.
\end{align*}
\end{thm}
The proof will be given in Section \ref{Sec:Proofs}.

In Theorem \ref{res:asymptS+} we obtain the asymptotic behavior of $S^+$ using Theorem \ref{res:exactS+} and ideas inspired from \cite{KDe92} 
and adapted to our framework (see also the discussion in Remark \ref{KD}). 
Before stating this result, we need to introduce few more notations.

For every $\alpha, \beta \in \mathcal{A}$ and $\ell \in \mathbb{N}$ we denote
$$G_{\alpha\beta}^{(\ell)} :=
\frac{u_{\beta}(\theta^*)}{u_{\alpha}(\theta^*)} e^{\theta^* \ell d}L_{\alpha\beta}^{(\ell)},\qquad
G_{\alpha\beta}{(\ell)} := \sum_{k=0}^{\ell} G_{\alpha\beta}^{(k)},\qquad
G_{\alpha\beta}{(\infty)} := \sum_{k=0}^{\infty} G_{\alpha\beta}^{(k)}.
$$
The matrix ${\bf G(\infty)}:=(G_{\alpha\beta}(\infty))_{\alpha, \beta}$ is stochastic, using Lemma \ref{EspeUsigma}; the subset $\mathcal{A}^+$ is a recurrent class, whereas the states in $\mathcal{A}\setminus \mathcal{A}^+$ are transient. The restriction of ${\bf G(\infty)}$ to $\mathcal{A}^+$ is stochastic and irreducible; let us denote $\tilde{w} > 0$ the corresponding stationary frequency vector. Define $w=(w_{\alpha})_{\alpha\in\mathcal{A}}$, with $w_{\alpha}=\tilde{w}_{\alpha} > 0$ for $\alpha\in\mathcal{A}^+$ and $w_{\alpha}=0$ for $\alpha \in \mathcal{A}\setminus \mathcal{A}^+$. The vector $w$ is invariant for $\bf G(\infty)$, i.e. $w{\bf G(\infty)}=w$.

\begin{rem}
Note that in Karlin and Dembo's Markovian model of \cite{KDe92} the matrix ${\bf G(\infty)}$ is positive, hence irreducible, thanks to their random scoring function and to their hypotheses recalled in Remark \ref{KD}.
\end{rem}

\begin{rem}
Note that the coefficients $L_{\alpha\beta}^{(k)}$ can be computed recursively (see Subsection \ref{Subsec:L}). In Subsection \ref{Subsec:FS+} we present in detail a recursive procedure for computing the \textit{cdf} $F_{S^+,\alpha}$, based on Theorem \ref{res:exactS+}. Note also that, for every $\alpha, \beta \in \mathcal{A}$, there are a finite number of $L_{\alpha\beta}^{(k)}$ different from zero. Therefore, there are a finite number of non-null terms in the sum defining $G_{\alpha \beta}(\infty)$.
\end{rem}

\begin{thm}[Asymptotic distribution of $S^+$]\label{res:asymptS+}
For every $\alpha\in {\cal A}$ we have
\begin{equation}\label{eq:cinf}
\lim_{k \rightarrow +\infty} \frac{e^{\theta^*{kd}}\P_{\alpha}(S^+>kd)}{u_{\alpha}(\theta^*)} = \frac{d}{c} \cdot \sum_{\gamma \in \mathcal{A}^+} \frac{w_{\gamma}}{u_{\gamma}(\theta^*)}\sum_{\ell \geq 0} (L_{\gamma}(\infty)-L_{\gamma}(\ell d)) e^{\theta^*\ell d}:= c(\infty),
\end{equation}
where
$w=(w_{\alpha})_{\alpha\in {\cal A}}$ is the stationary frequency vector of the matrix $\bf G(\infty)$   
and 
$$c:=\sum_{\gamma,\beta \in \mathcal{A}^+}  \frac{w_{\gamma}}{u_{\gamma}(\theta^*)}u_{\beta}(\theta^*) \sum_{\ell \geq 0} \ell d \cdot e^{\theta^*{\ell d}} \ L^{(\ell)}_{\gamma \beta}.$$
\end{thm}
The proof is deferred to Section \ref{Sec:Proofs}.

\begin{rem}\label{rem:cinf}
Note that there are a finite number of non-null terms in the above sums over $\ell$.
We also have the following alternative expression for $c(\infty)$:
$$
c(\infty) = \frac{d}{c(e^{\theta^*d}-1)} \cdot \sum_{\gamma \in \mathcal{A}^+} \frac{w_{\gamma}}{u_{\gamma}(\theta^*)} \left\{
\E_\gamma\left[e^{\theta^* S_{\sigma^+}}; \sigma^+ < \infty\right] - L_\gamma(\infty)
\right\}.
$$
Indeed, by the summation by parts formula
$$
\sum_{\ell=m}^k f_\ell (g_{\ell+1}-g_\ell) = f_{k+1} g_{k+1} - f_m g_m -  \sum_{\ell=m}^k  (f_{\ell+1}-f_\ell)g_{\ell+1},
$$
we obtain
\begin{align*}
& \sum_{\ell = 0}^{\infty} (L_{\gamma}(\infty)-L_{\gamma}(\ell d)) e^{\theta^*\ell d} 
= \frac{1}{e^{\theta^*d}-1} \sum_{\ell = 0}^{\infty} (L_{\gamma}(\infty)-L_{\gamma}(\ell d)) \left(e^{\theta^*(\ell+1) d} - e^{\theta^*\ell d} \right)\\
& = \frac{1}{e^{\theta^*d}-1}\\
&\quad\times \left\{ \lim_{k \to \infty} (L_{\gamma}(\infty)-L_{\gamma}(k d)) e^{\theta^*k d} - L_{\gamma}(\infty) - \sum_{\ell = 0}^{\infty} (L_{\gamma}(\ell d)-L_{\gamma}((\ell+1) d)) e^{\theta^*(\ell+1) d}  \right\}\\
& = \frac{1}{e^{\theta^*d}-1} \left\{  - L_{\gamma}(\infty) + \sum_{\ell = 0}^{\infty} e^{\theta^*(\ell+1) d} \ \P_{\gamma}(S_{\sigma^+} = (\ell+1) d, \ \sigma^+<\infty) \right\}\\
& = \frac{1}{e^{\theta^*d}-1}  \left\{
\E_\gamma\left[e^{\theta^* S_{\sigma^+}}; \sigma^+ < \infty\right] - L_\gamma(\infty)
\right\}.
\end{align*}
\end{rem}

Before stating the next results, let us denote for every integer $\ell < 0$ and $\alpha, \beta \in \mathcal{A}$, 
$$Q_{\alpha\beta}^{(\ell)}:=\P_{\alpha}(S_{\sigma^-}=\ell d,A_{\sigma^-}=\beta).$$
Note that $Q_{\alpha\beta}^{(\ell)} = 0$ for $\beta \in  \mathcal{A} \setminus \mathcal{A}^-$.
In Section \ref{sec:comp} we give a recursive computational method for obtaining these quantities.

Using Theorem \ref{res:asymptS+} we obtain the following
\begin{thm}[Asymptotic distribution of $Q_1$]\label{res:Q1}
We have the following asymptotic result on the distribution of the 
height of the first non-negative excursion: for every $\alpha \in \mathcal{A}$ we have
\begin{equation}\label{eq:Q1}
\P_{\alpha}(Q_1>kd)\underset{k\rightarrow\infty}{\sim}\P_{\alpha}(S^+>kd)-\sum_{\ell<0}\sum_{\beta \in \mathcal{A}^-}\P_{\beta}\left (S^+>(k-\ell)d\right )\cdot Q_{\alpha\beta}^{(\ell)}.
\end{equation}
\end{thm}
The proof will be given in Section \ref{Sec:Proofs}.


Using now Theorems \ref{res:asymptS+} and \ref{res:Q1} we finally obtain the following result on the asymptotic distribution of the local score $M_n$ for a sequence of length $n$.
\begin{thm}[Asymptotic distribution of the local score $M_n$]\label{res:Mn}
For every $\alpha \in \mathcal{A}$:
\begin{align}\label{eq:Mn}
\P_\alpha  &\left(M_n\leq \frac{\log (n)}{\theta^*}+x\right )
 \underset{n\rightarrow\infty}{\sim} \exp\left \{-\frac{n}{A^*}\sum_{\beta \in \mathcal{A}^-} z_{\beta} \P_{\beta}\left (S^+> \left\lfloor\frac{\log(n)}{\theta^*}+x\right\rfloor\right )\right \} \nonumber \\
& \ \ \ \ \ \ \ \ \times\exp\left \{ \frac{n}{A^*}
\sum_{k < 0}\sum_{\gamma \in \mathcal{A}^-} \P_{\gamma}\left(S^+>\left\lfloor \frac{\log(n)}{\theta^*}+x \right \rfloor -kd\right) \cdot\sum_{\beta \in \mathcal{A}^-}z_{\beta}Q_{\beta\gamma}^{(k)}
\right \},
\end{align}
where $z=(z_{\alpha})_{\alpha\in {\cal A}}$ is the invariant probability measure of the matrix ${\bf Q}$ defined in Subsection \ref{subsec:notdef} and
$$A^* := \lim_{m\rightarrow +\infty} \frac{K_m}{m} =  \frac{1}{\E(f(A))}\sum_{\beta \in \mathcal{A}^-} z_{\beta}  \E_\beta(S_{\sigma^-})\mbox{ a.s.}$$
\end{thm}

\begin{rem}
\begin{itemize}
\item Note that the asymptotic equivalent in Equation (\ref{eq:Mn}) does not depend on the initial state $\alpha$.
\item We recall, for comparison, the asymptotic result of  \cite{KDe92} (Equation (1.27)) for the distribution of $M_n$:
\begin{equation}\label{Res:AppxMnKDe}
\lim_{n\rightarrow +\infty}\P_\alpha \left(M_n\leq \frac{\log (n)}{\theta^*}+x\right )
=\exp\left ( -K^*\exp(-\theta^*)\right ),
\end{equation}
with 
$K^*=v(\infty)\cdot c(\infty)$, where $c(\infty)$ given in Theorem \ref{res:asymptS+} is related to the defective distribution of the first positive partial sum $S_{\sigma^+}$ (see also Remark \ref{rem:cinf})
and $v(\infty)$ is related to the distribution of the first negative partial sum $S_{\sigma^-}$ (see Equations (5.1) and (5.2) of \cite{KDe92} for more details). A more explicit formula for $K^*$ is given in Subsection \ref{subsec:simplecase} for an application in a simple case. 

\item Note that our asymptotic equivalent in Equation (\ref{eq:Mn}) keeps the dependence on $n$, whereas the approximation derived from Equation (\ref{Res:AppxMnKDe}) does not. 

\end{itemize}
\end{rem}


\section{Proofs of the main results}\label{Sec:Proofs}

\subsection{Proof of Theorem \ref{res:exactS+}}

We have
\begin{align*}
F_{S^+,\alpha}(\ell d)&= \P_\alpha(\sigma^+ = \infty) + \P_\alpha(S^+ \leq \ell d, \sigma^+ < \infty)\\
&=1 - L_{\alpha}(\infty) + \sum_{\beta \in \mathcal{A}^+} \sum_{k=1}^{\ell} \P_\alpha(S^+ \leq \ell d, \sigma^+ < \infty, S_{\sigma^+} = k d, A_{\sigma^+} = \beta)\\
&= 1 - L_{\alpha}(\infty) + \sum_{\beta \in \mathcal{A}^+} \sum_{k=1}^{\ell} L^{(k)}_{\alpha \beta} \ \P_\alpha(S^+ \leq \ell d \ | \sigma^+ < \infty, S_{\sigma^+} = k d, A_{\sigma^+} = \beta).
\end{align*}
The last probability can further be written
$$
\P_\alpha(S^+ - S_{\sigma^+} \leq (\ell-k)d \ | \sigma^+ < \infty, S_{\sigma^+} = k d, A_{\sigma^+} = \beta)
= \P_\beta(S^+ \leq (\ell-k)d),
$$
by the strong Markov property applied to the stopping time $\sigma^+$.
The stated result easily follows. $\hfill\square$

\subsection{Proof of Theorem \ref{res:asymptS+}}

We first prove some preliminary lemmas. 

\begin{lem}\label{Splus}
We have $\lim_{k \to \infty} \P_\alpha(S^+ > kd)= 0$ for every $\alpha \in \mathcal{A}$.
\end{lem}

\begin{proof} 
With $F_{S^+,\alpha}$ defined in Theorem \ref{res:exactS+}, we introduce for every $\alpha$ and $\ell \geq 0$:
$$
b_{\alpha}(\ell d) := \frac{1-F_{S^+,\alpha}(\ell d)}{u_{\alpha}(\theta^*)}e^{\theta^*\ell d}, \ \ a_{\alpha}(\ell d) := \frac{L_{\alpha}(\infty) - L_{\alpha}(\ell d)}{u_{\alpha}(\theta^*)}e^{\theta^*\ell d}.
$$
Theorem \ref{res:exactS+} allows to obtain the following renewal system for the family $(b_{\alpha})_{\alpha\in\mathcal{A}}$: 
$$\forall\ell > 0,\forall\alpha \in \mathcal{A},\quad
b_{\alpha}(\ell d)= a_{\alpha}(\ell d) + \sum_{\beta} \sum_{k = 0}^{\ell} b_{\beta}((\ell-k)d) G^{(k)}_{\alpha \beta}. 
$$
Since the restriction of ${\bf \tilde G(\infty)}$ of ${\bf G(\infty)}$ to $\mathcal{A}^+$ is stochastic, its spectral radius equals 1 and a corresponding right eigenvector is the vector having all components equal to 1; a left eigenvector is the stationary frequency vector $\tilde w > 0$.


\noindent \textit{Step 1}: For every $\alpha \in \mathcal{A}^+$, a direct application of Theorem 2.2 of Athreya and Murthy \cite{AMu75} gives the formula in Equation \ref{eq:cinf} for the limit $c(\infty)$ of $b_{\alpha}(\ell d)$ when $\ell \to \infty$, which implies the stated result.\\

\noindent \textit{Step 2}: Consider now $\alpha \notin \mathcal{A}^+$. By Theorem \ref{res:exactS+} we have
$$
\P_\alpha(S^+ > \ell d) = L_{\alpha}(\infty) - \sum_{\beta \in \mathcal{A}^+} \sum_{k=1}^{\ell} L^{(k)}_{\alpha \beta} \ \left\{1-\P_\beta(S^+ > (\ell-k) d)\right\}.
$$
Since $\P_\beta(S^+ > (\ell-k) d) = 1$ for $k > \ell$ and $L_{\alpha}(\infty) = \sum_{\beta \in \mathcal{A}^+} \sum_{k=1}^{\infty} L^{(k)}_{\alpha \beta}$, we deduce
\begin{equation}\label{eq:S+}
\P_\alpha(S^+ > \ell d) = \sum_{\beta \in \mathcal{A}^+} \sum_{k=1}^{\infty} L^{(k)}_{\alpha \beta} \ \P_\beta(S^+ > (\ell-k) d).
\end{equation}
Note that for fixed $\alpha$ and $\beta$, there are a finite number of non-null terms in the above sum over $k$. Using the fact that for fixed $\beta \in \mathcal{A}^+$ and $k \geq 1$ we have $\P_\beta(S^+ > (\ell-k) d) \longrightarrow 0$ when $\ell \to \infty$, as shown previously in Step 1, the stated result follows.\end{proof}

\begin{lem}\label{UmMartingale}
Let $\theta > 0$. With $u(\theta)$ defined in Subsection \ref{subsec:notdef}, the sequence of random variables $(U_m(\theta))_{m \geq 0}$ defined by $U_0(\theta):=1$ and
$$U_m(\theta):=\prod_{i=0}^{m-1}\left [ \frac{\exp(\theta f(A_{i+1}))}{u_{A_i}(\theta)}\cdot\frac{u_{A_{i+1}}(\theta)}{\rho(\theta)}\right ]=\frac{\exp(\theta S_m)u_{A_m}(\theta)}{\rho(\theta)^m u_{A_0}(\theta)} \ \ \text{, for } m \geq 1 $$
is a martingale with respect to the canonical filtration  $\mathcal{F}_m = \sigma(A_0,\ldots, A_m).$
\end{lem}

\begin{proof}
We have
$$U_{m+1}(\theta)=U_m(\theta) \frac{\exp(\theta f(A_{m+1})) u_{A_{m+1}}(\theta)}{u_{A_m}(\theta)\rho(\theta)}.$$
Since $U_m(\theta)$ and $u_{A_m} (\theta)$ are measurable with respect to $\mathcal{F}_m$, we have
$$
\E[U_{m+1}(\theta) | \mathcal{F}_m] = U_m(\theta) \frac{\E[\exp(\theta f(A_{m+1})) u_{A_{m+1}}(\theta) | \mathcal{F}_m] }{u_{A_m} (\theta)\rho(\theta)}.
$$
By the Markov property we further have  
$$
\E[\exp(\theta f(A_{m+1})) u_{A_{m+1}}(\theta) | \mathcal{F}_m]= \E[\exp(\theta f(A_{m+1})) u_{A_{m+1}}(\theta) | A_m]$$
and by definition of $u(\theta)$,
\begin{align*}
\E[\exp(\theta f(A_{m+1})) u_{A_{m+1}}(\theta) | A_m = \alpha] &= \sum_{\beta} \exp(\theta f(\beta)) u_{\beta}(\theta) p_{\alpha \beta}\\
& = (\Phi(\theta)u(\theta))_{\alpha}=u_{\alpha}(\theta)\rho(\theta).
\end{align*}
We deduce
$$\E[\exp(\theta f(A_{m+1})) u_{A_{m+1}}(\theta) | A_m] = u_{A_m}(\theta)\rho(\theta),$$
hence
$
\E[U_{m+1}(\theta) | \mathcal{F}_m]=U_{m}(\theta),
$
which finishes the proof.\end{proof}

\begin{lem}
\label{EspeUsigma}
With $\theta^*$ defined at the end of Subsection  \ref{subsec:notdef} we have
\begin{equation}
\forall\alpha \in \mathcal{A}: \ \quad \frac{1}{u_\alpha(\theta^*)} \sum_{\beta \in \mathcal{A}^+} \sum_{\ell = 1}^{\infty} L^{(\ell)}_{\alpha \beta} \ e^{\theta^* \ell d} \ u_\beta(\theta^*) = 1.
\end{equation}
\end{lem}

\begin{proof}
The proof uses Lemma \ref{Splus} and ideas inspired from \cite{KDe92} (Lemma 4.2). First note that the above equation is equivalent to
$$\E_{\alpha}[U_{\sigma^+}(\theta^*);\sigma^+<\infty]=1,$$
with $U_m(\theta)$ defined in Lemma \ref{UmMartingale}. By applying the optional sampling theorem to the bounded stopping time $\tau_n := \min(\sigma^+,n)$ and to the martingale $(U_m(\theta^*))_m$, we obtain
$$
1=\E_\alpha[U_0(\theta^*)]=\E_\alpha[U_{\tau_n}(\theta^*)] = \E_\alpha[U_{\sigma^+}(\theta^*); \sigma^+ \leq n] + \E_\alpha[U_{n}(\theta^*); \sigma^+ > n].
$$ 
We will show that $\E_\alpha[U_{n}(\theta^*); \sigma^+ > n] \longrightarrow 0$ when $n \to \infty$. Passing to the limit in the previous relation will then give the desired result. Since $\rho(\theta^*)=1$, we have
$$U_n(\theta^*) =\frac{\exp(\theta^* S_n)u_{A_n}(\theta^*)}{u_{A_0}(\theta^*)}$$
and it suffices to show that $\lim_{n \to \infty}\E_\alpha[\exp(\theta^* S_n); \sigma^+ > n] = 0$.\par
For a fixed $a > 0$ we can write
\begin{align}\label{esperance}
\E_\alpha[\exp(\theta^* S_n); \sigma^+ > n] &= \E_\alpha[\exp(\theta^* S_n); \sigma^+ > n, \ \exists k \leq n:S_k \leq -2a] \nonumber  \\
&\ \ \ +  \E_\alpha[\exp(\theta^* S_n); \sigma^+ > n, -2a \leq S_k \leq 0, \ \forall 0 \leq k \leq n].
\end{align}
The first expectation in the right-hand side of Equation (\ref{esperance}) can further be bounded as follows:
\begin{align}\label{esperance1}
\E_\alpha[\exp(\theta^* S_n);  \sigma^+  & >  n , \ \exists k \leq n:S_k \leq -2a]  \leq  \E_\alpha[\exp(\theta^* S_n); \sigma^+ > n, S_n \leq -a] \nonumber \\
& + \E_\alpha[\exp(\theta^* S_n); \sigma^+ > n, S_n > -a, \ \exists k <  n:S_k \leq -2a].
\end{align}
We obviously have
\begin{equation}\label{esperance2}
\E_\alpha[\exp(\theta^* S_n); \sigma^+ > n, S_n \leq -a] \leq \exp(-\theta^* a).
\end{equation}

Let us further define the stopping time $T := \inf\{k \geq 1 : S_k \leq -2a\}$. Note that $T < \infty$ \textit{a.s.} since $S_n \longrightarrow -\infty$ \textit{a.s.} when $n \to \infty$. Indeed, by the ergodic theorem we have $S_n/n \longrightarrow \E[f(A)] < 0$ when $n \to \infty$. Therefore we have
\begin{align*}
\E_\alpha[\exp(\theta^* S_n)& ;  \sigma^+ > n, S_n > -a,  \ \exists k <  n:S_k \leq -2a] \leq \P_\alpha(T \leq n, S_n > -a)\\
&= \sum_{\beta \in \mathcal{A}^-} \P_\alpha(T \leq n, S_n > -a \ | A_T = \beta)\P_\alpha(A_T = \beta)\\
& \leq \sum_{\beta \in \mathcal{A}^-} \P_\alpha(S_n - S_T > a \ | A_T = \beta)\P_\alpha(A_T = \beta)\\
& \leq \sum_{\beta \in \mathcal{A}^-} \P_\beta(S^+ > a) \P_\alpha(A_T = \beta),
\end{align*}
by the strong Markov property.
For every $a > 0$ we thus have
\begin{equation}\label{esperance3}
\limsup_{n \to \infty} \E_\alpha[\exp(\theta^* S_n); \sigma^+ > n, S_n > -a, \ \exists k <  n:S_k \leq -2a] \leq \sum_{\beta \in \mathcal{A}^-} \P_\beta(S^+ > a).
\end{equation}

Considering the second expectation in the right-hand side of Equation (\ref{esperance}), we have
\begin{equation}\label{esperance4}
\lim_{n \to \infty} \P_\alpha(-2a \leq S_k \leq 0, \  \forall 0 \leq k \leq n) =  \P_\alpha(-2a \leq S_k \leq 0, \  \forall  k \geq 0) = 0, 
\end{equation}
again since $S_n \longrightarrow -\infty$ \textit{a.s.} when $n \to \infty$.

Equations (\ref{esperance}),(\ref{esperance1}),(\ref{esperance2}),(\ref{esperance3}) and (\ref{esperance4}) imply that for every $a > 0$ we have
$$
\limsup_{n \to \infty} \E_\alpha[\exp(\theta^* S_n);  \sigma^+ > n] \leq  \exp(-\theta^*a) +  \sum_{\beta \in \mathcal{A}^-} \P_\beta(S^+ > a).
$$
Using Lemma \ref{Splus} and taking $a \to \infty$ we obtain $\lim_{n \to \infty}\E_\alpha[\exp(\theta^* S_n); \sigma^+ > n] = 0.$
\end{proof}

We are now ready to prove the Theorem  \ref{res:asymptS+}.

\noindent {\bf Proof of Theorem \ref{res:asymptS+}:}\\
%
%
For $\alpha \in \mathcal{A}^+$ the formula has been already shown in Step 1 of the proof of Lemma \ref{Splus}.
For $\alpha \notin \mathcal{A}^+$ we will prove the stated formula using Theorem \ref{res:exactS+}. From Equation (\ref{eq:S+}), we have
\begin{equation*}
\P_\alpha(S^+ > \ell d) = \sum_{\beta \in \mathcal{A}^+} \sum_{k=1}^{\infty} L^{(k)}_{\alpha \beta} \ \P_\beta(S^+ > (\ell-k) d),
\end{equation*}
hence
\begin{equation*}
\frac{e^{\theta^*{\ell d}}\P_{\alpha}(S^+>\ell d)}{u_{\alpha}(\theta^*)} = \sum_{\beta \in \mathcal{A}^+} \sum_{k=1}^{\infty}  \frac{e^{\theta^*{(\ell-k)d}}\P_{\beta}(S^+>(\ell-k)d)}{u_{\beta}(\theta^*)} L^{(k)}_{\alpha \beta} e^{\theta^*{kd}} \frac{u_{\beta}(\theta^*)}{u_{\alpha}(\theta^*)}.
\end{equation*}
Note that for every $\alpha$ and $\beta$ there are a finite number of non-null terms in the above sum over $k$. Moreover, as shown in Lemma \ref{Splus}
$$\forall\beta \in \mathcal{A}^+,\ \forall k\geq 0:\ \quad 
 \frac{e^{\theta^*{(\ell-k)d}}\P_{\beta}(S^+>(\ell-k)d)}{u_{\beta}(\theta^*)} \underset{\ell \to \infty}{\longrightarrow} c(\infty).
$$
We finally obtain that
$$
\lim_{\ell \rightarrow +\infty}  \frac{e^{\theta^*{\ell d}}\P_{\alpha}(S^+> \ell d)}{u_{\alpha}(\theta^*)} = \frac{c(\infty)}{u_\alpha(\theta^*)} \sum_{\beta \in \mathcal{A}^+} \sum_{k = 1}^{\infty} L^{(k)}_{\alpha \beta} \ e^{\theta^* k d} \ u_\beta(\theta^*),
$$
which equals $c(\infty)$ as desired, by Lemma \ref{EspeUsigma}.

\subsection{Proof of Theorem \ref{res:Q1}}
Since $S^+ \geq Q_1$, for every $\alpha \in \mathcal{A}$ we have
$$
\P_{\alpha}(S^+>kd) = \P_{\alpha}(Q_1>kd) + \P_{\alpha}(S^+>kd, Q_1 \leq kd).
$$
We will further decompose the last probability with respect to the values taken by $S_{\sigma^-}$ and $A_{\sigma^-}$, as follows:
\begin{align*}
& \P_{\alpha}(S^+>kd, Q_1 \leq kd) = \sum_{\ell<0}\sum_{\beta \in \mathcal{A}^-}\P_{\alpha}(S^+>kd, Q_1 \leq kd, S_{\sigma^-}=\ell d,A_{\sigma^-}=\beta)\\
&=  \sum_{\ell<0}\sum_{\beta \in \mathcal{A}^-}\P_{\alpha}(S^+ -S_{\sigma^-} >(k-\ell)d \ | A_{\sigma^-}=\beta, Q_1 \leq kd, S_{\sigma^-}=\ell d )  \\
& \ \ \ \  \times \P_{\alpha}(Q_1 \leq kd, S_{\sigma^-}=\ell d, A_{\sigma^-}=\beta)\\
&=  \sum_{\ell<0}\sum_{\beta \in \mathcal{A}^-}\P_{\beta}(S^+ >(k-\ell)d ) \cdot  \left\{Q_{\alpha\beta}^{(\ell)}- \P_{\alpha}(Q_1 > kd, S_{\sigma^-}=\ell d, A_{\sigma^-}=\beta)\right \},
\end{align*}
by applying the strong Markov property to the stopping time $\sigma^-$. We thus obtain
\begin{align*}
& \P_{\alpha}(S^+>kd)-\sum_{\ell <0}\sum_{\beta \in \mathcal{A}^-}\P_{\beta}(S^+>(k-\ell )d)\cdot Q_{\alpha\beta}^{(\ell )} - \P_{\alpha}(Q_1>kd) \\
& = - \sum_{\ell <0}\sum_{\beta \in \mathcal{A}^-}\P_{\beta}(S^+ >(k-\ell )d ) \ \P_{\alpha}(Q_1 > kd, S_{\sigma^-}=\ell d, A_{\sigma^-}=\beta).
\end{align*}
By Theorem \ref{res:asymptS+} we have $\P_{\beta}(S^+ >kd ) = O(e^{-\theta^*kd})$ as $k \to \infty$, for every $\beta \in \mathcal{A}^-$, from which we deduce that the left-hand side of the previous equation is $o(\P_{\alpha}(Q_1 > kd))$ when $k \to \infty$.
The stated result then easily follows. $\hfill\square$

\subsection{Proof of Theorem \ref{res:Mn}}

We will first prove some useful lemmas.

\begin{lem}\label{lem:rho'}
We have
$\rho'(0)=\E[f(A)]<0.$
\end{lem}

\begin{proof}
By the fact that $\rho(\theta)$ is an eigenvalue of the matrix $\Phi(\theta)$ with corresponding eigenvector $u(\theta)$,  we have
$$\rho(\theta)u_{\alpha}(\theta)=\left (\Phi(\theta)u(\theta)\right )_{\alpha}
=\sum_\beta p_{\alpha\beta}e^{\theta f(\beta)}u_\beta(\theta).$$
When derivating the previous relation with respect to $\theta$ we obtain
$$\frac{d}{d \theta}(\rho (\theta)u_\alpha(\theta))=
\sum_\beta p_{\alpha\beta} \left ( f(\beta)e^{\theta f(\beta)}u_\beta(\theta)+e^{\theta f(\beta)}u'_{\beta}(\theta)\right ).$$
We have $\rho(0)=1$ et $u(0)=^t(1/r,\dots,1/r)$. For $\theta=0$, we then get
\begin{equation}\label{eq:3}
\left. \sum_\alpha \pi_{\alpha} \frac{d}{d \theta}(\rho (\theta)u_\alpha(\theta)) \right |_{\theta=0} =\frac{1}{r}\E[f(A)]+\sum_{\alpha,\beta}\pi_{\alpha} p_{\alpha\beta}u'_{\beta}(0)
=\frac{1}{r}\E[f(A)]+\sum_{\beta}\pi_{\beta}u'_{\beta}(0).
\end{equation}

On the other hand,
$$ \sum_\alpha \pi_{\alpha} \frac{d}{d \theta}(\rho (\theta)u_\alpha(\theta))=
\frac{d}{d \theta}\left(\sum_\alpha \pi_{\alpha} \rho (\theta)u_\alpha(\theta)\right) =\rho'(\theta)\sum_\alpha \pi_{\alpha} u_\alpha(\theta)+\rho(\theta)\sum_\alpha \pi_{\alpha} u'_\alpha(\theta).$$
For $\theta=0$ we get
\begin{equation}\label{eq:4}
\left. \sum_\alpha \pi_{\alpha} \frac{d}{d \theta}(\rho (\theta)u_\alpha(\theta))  \right |_{\theta=0}=\frac{\rho'(0)}{r}+\rho(0)\cdot\sum_{\alpha}\pi_{\alpha}u'_{\alpha}(0).
\end{equation}
From Equations (\ref{eq:3}) and (\ref{eq:4}) we deduce
$$\frac{\rho'(0)}{r}+\sum_{\alpha}\pi_{\alpha}u'_{\alpha}(0)=\frac{1}{r}\E[f(A)]+\sum_\beta\pi_{\beta}u'_{\beta}(0),$$
from which the stated result easily follows.
\end{proof}



\begin{lem}\label{lem:DKa91}
There exist ${\cal I}>0$ and $n_0\geq 0$ such that $\forall n\geq n_0$, $\P_{\alpha}(S_n\geq 0)\leq \exp(-n{\cal I})$ for every $\alpha \in \mathcal{A}$.
\end{lem}
\begin{proof}
By a large deviation principle for additive functionals of Markov chains (see Theorem 3.1.2. in \cite{DKa91b}) we have
$$\limsup_{n\rightarrow +\infty}\frac{1}{n}\log\left ( \P_{\alpha}\left(\frac{S_n}{n} \in \Gamma\right)\right)
\leq -{\cal I},$$
with $\Gamma=[0,+\infty )$ and  ${\cal I}=\inf_{x\in \bar{\Gamma}} \sup_{\theta\in\R} (\theta x-\log \rho(\theta))$. 
Since $\cal A$ is finite, it remains to prove that ${\cal I}>0$. 

For every $x \geq 0$, let us denote $g_x(\theta):=\theta x-\log \rho(\theta)$ and $I(x):=\sup_{\theta\in\R} g_x(\theta)$.
We will first show that $I(x)=\sup_{\theta\in\R^+} g_x(\theta)$.
Indeed, we have $g'_x(\theta)= x -  \rho'(\theta)/\rho(\theta)$. By the strict convexity property of $\rho$ (see \cite{DKa91b, KOs87}) and the fact that $\rho '(0)=\E[f]<0$ (by Lemma \ref{lem:rho'}), we deduce that $\rho'(\theta)<0$ for every $\theta \leq 0$, implying that  $g'_x(\theta)> x \geq 0$ for $\theta \leq 0$. The function $g'_x$ is therefore increasing on $\R^-$, and hence $I(x)=\sup_{\theta\in\R^+} g_x(\theta)$.

As a consequence, we deduce that $x \mapsto I(x)$ is non-decreasing on $\R^+$.
We thus obtain ${\cal I}=\inf_{x\in \R^+}I(x)=I(0)$. 

Further, we have $I(0)=\sup_{\theta\in\R} \left(-\log \rho(\theta)\right)=-\inf_{\theta\in \R^+} \log(\rho(\theta))$. Using again the fact that $\rho'(0)<0$ (Lemma \ref{lem:rho'}), the strict convexity of $\rho$ and the fact that $\rho(0)=\rho(\theta^*)=1$, we finally obtain ${\cal I}=-\log\left(\inf_{\theta\in \R^+} \rho(\theta)\right) >-\log \rho(0)=0$. The statement then follows.
\end{proof}

\begin{lem}\label{lem:Esp_K1}
We have $\E_\alpha(K_1) < \infty$ for every $\alpha \in \mathcal{A}$.
\end{lem}
\begin{proof}
Note that $\P_\alpha(K_1>n)\leq\P_\alpha(S_n\geq 0)$. With $n_0 \in \N$ and ${\cal I}>0$ defined in  Lemma \ref{lem:DKa91}, using a well-known formula for the expectation, we get 
$$
\E_\alpha[K_1]=\sum_{n\geq 0}\P_\alpha(K_1>n)\leq\sum_{n\geq 0}\P(S_n\geq 0)\leq C + \sum_{n\geq n_0}\exp(-n{\cal I}),
$$
where $C >0$ is a constant. The statement easily follows.
\end{proof}

\begin{lem}
\label{Km}
We have
$$\lim_{m\rightarrow +\infty} \frac{K_m}{m} =  \sum_{\beta} z_\beta \E_\beta(K_1)  \mbox{ a.s.}$$
\end{lem}

\begin{proof}
Recall that $K_1 = \sigma^-$. We can write
\begin{equation}\label{K}
\frac{K_m}{m} = \frac{K_1}{m} + \frac{1}{m}\sum_{i=2}^m (K_i - K_{i-1}) = \frac{K_1}{m} + \sum_\beta  \frac{1}{m}\sum_{i=2}^m (K_i - K_{i-1}) \textbf{1}_{\{A_{K_{i-1}}=\beta\}}. 
\end{equation}
First note that $\displaystyle \frac{K_1}{m} \to 0 \ a.s.$ when $m \to \infty$, since $K_1 < +\infty \ a.s.$ By the strong Markov property we have that, conditionally on $(A_{K_{i-1}})_{i \geq 2}$, the random variables $(K_i-K_{i-1})_{i \geq 2}$ are all independent and the distribution of $K_i-K_{i-1}$ depends only on $A_{K_{i-1}}$ and we have
$\P(K_i-K_{i-1}=\ell\  |  A_{K_{i-1}} = \alpha) = \P_\alpha(K_1 = \ell)$.
Therefore, the couples $Y_i := (A_{K_{i-1}}, K_i-K_{i-1}), i \geq 2$ form a Markov chain on $\mathcal{A}^- \times \N$, with transition probabilities
$
\P(Y_i = (\beta,\ell) \ | Y_{i-1} = (\alpha,k)) = q_{\alpha \beta} \P_\beta(K_1 = \ell)$.
Recall that the restriction ${\bf \tilde Q}$ of the matrix ${\bf  Q}$ to the subspace $\mathcal{A}^-$ is irreducible.
Therefore, the Markov chain $(Y_i)_i$ is also irreducible and we can show that
$\pi(\alpha,k):=z_\alpha \P_\alpha(K_1 = k)$ is its invariant distribution. Indeed, since $z$ is invariant for ${\bf Q}$, we easily deduce that
$$
\sum_{\alpha,k} \pi(\alpha,k) \cdot q_{\alpha \beta} \P_{\beta}(K_1=\ell) = \pi(\beta,\ell).
$$

For fixed $\beta$, when applying the ergodic theorem to the Markov chain $(Y_i)_i$ and the function $\varphi_\beta(\alpha, k) := k \textbf{1}_{\{\alpha=\beta\}}$, we deduce
$$
\frac{1}{m}\sum_{i=2}^m (K_i - K_{i-1}) \textbf{1}_{\{A_{K_{i-1}}=\beta\}} \longrightarrow \sum_{\alpha,k} \varphi_\beta(\alpha,k)\pi(\alpha,k) =  z_\beta \E_\beta(K_1) \ \ a.s. 
$$
when $m \to \infty$.
Taking the sum over $\beta$ and using the relation (\ref{K}) gives the desired result. 
\end{proof}

\noindent {\bf Proof of Theorem \ref{res:Mn}:}\\
The proof is inspired from \cite{KDe92}.

Given $(A_{K_i})_{i \geq 0}$, the random variables $(Q_i)_{i \geq 1}$ are independent and the \textit{cdf} of $Q_i$ is $F_{A_{K_{i-1}}A_{K_i}}$.  Therefore
\begin{eqnarray*}
 \P_{\alpha}\left ( M_{K_m}\leq y  \right)&=& \E_\alpha\left [ \prod_{i=1}^mF_{A_{K_{i-1}}A_{K_i}}(y)\right ]\\
 &=&  \E_\alpha\left [ \exp\left \{\sum_{\beta,\gamma \in \mathcal{A}} m \psi_{\beta\gamma}(m)\log(F_{\beta\gamma}(y))\right \}\right ],
\end{eqnarray*} 
with $\psi_{\beta\gamma}(m):=\#\{i:1 \leq i\leqp m,A_{K_{i-1}}=\beta,A_{K_{i}}=\gamma\}/m$. Given that $A_0 = \alpha \in \mathcal{A}^-$, the states $(A_{K_i})_{i \geq 0}$ form an irreducible Markov chain on $\mathcal{A}^-$ of transition matrix ${\bf \tilde Q}=(q_{ \beta \gamma})_{\beta, \gamma \in \mathcal{A}^-}$ and stationary frequency vector $\tilde z = (z_\beta)_{\beta \in \mathcal{A}^-} > 0$. 

Consequently, for $\beta, \gamma \in \mathcal{A}^-$ the ergodic theorem implies that $\psi_{\beta\gamma}(m)\longrightarrow z_{\beta}q_{\beta\gamma}$ \textit{a.s.} when $m\rightarrow\infty$. On the other hand, for any $\alpha \in \mathcal{A}$, if $\beta \in \mathcal{A} \setminus \mathcal{A}^-$, then $\psi_{\beta\gamma}(m)\leqp 1/m$ and thus $\psi_{\beta\gamma}(m)\longrightarrow 0$ \textit{a.s.} when $m\rightarrow\infty$, for any $\gamma \in  \mathcal{A}$. With $z_{\beta} = 0$ for $\beta \in \mathcal{A} \setminus \mathcal{A}^-$, we thus have
$\psi_{\beta\gamma}(m)\longrightarrow z_{\beta}q_{\beta\gamma}$ \textit{a.s.} when $m\rightarrow\infty$, for every $\beta, \gamma \in \mathcal{A}$. 

Denoting $d_{\beta\gamma}(m):=m\left[1-F_{\beta\gamma}\left(\frac{\log m}{\theta^*}+x\right)\right]$ 
and using the fact that $d_{\beta\gamma}(m)$ are uniformly bounded in $\beta$ and $\gamma$, we have
\begin{eqnarray*}
\underset{m\rightarrow\infty}{\lim}\P_{\alpha}\left(M_{K_m}\leq\frac{\log m}{\theta^*}+x\right)
&=&\underset{m\rightarrow\infty}{\lim}\E_{\alpha}\left [\exp\left(-\sum_{\beta,\gamma \in \mathcal{A}}\psi_{\beta\gamma}(m)d_{\beta\gamma}(m)\right )\right ] \\
&=&\underset{m\rightarrow\infty}{\lim}\exp\left(-\sum_{\beta,\gamma \in \mathcal{A}}z_{\beta}q_{\beta\gamma}d_{\beta\gamma}(m)\right ).
\end{eqnarray*}
Since
$$\sum_{\gamma \in \mathcal{A}}q_{\beta\gamma}d_{\beta\gamma}(m)=m\left[1-F_{\beta}\left(\frac{\log m}{\theta^*}+x\right)\right],$$
$$\underset{m\rightarrow\infty}{\lim}\P_{\alpha}\left(M_{K_m}\leq \frac{\log m}{\theta^*}+x\right)
=\underset{m\rightarrow\infty}{\lim}\exp \left(-m \sum_{\beta \in \mathcal{A}^-} z_{\beta}\left [1-F_{\beta}\left(\frac{\log m}{\theta^*}+x\right )\right]\right).$$
But $$1-F_{\beta}\left(\frac{\log m}{\theta^*}+x\right)
=\P_{\beta}\left(Q_1> \frac{\log(m)}{\theta^*}+x\right)
=\P_{\beta}\left(Q_1> \left \lfloor\frac{\log(m)}{\theta^*}+x\right\rfloor\right),$$
and hence, with $\displaystyle y=y(m):=\frac{\log(m)}{\theta^*}+x$ we get, using Theorem \ref{res:Q1}:
\begin{align*}
1-F_{\beta}\left (\frac{\log m}{\theta^*}+x\right )& \underset{m\rightarrow\infty}{\sim}
\P_{\beta}\left(S^+> \left \lfloor\frac{\log(m)}{\theta^*}+x\right \rfloor\right)\\
&\ \ \ \ \ \ \ -\sum_{k < 0}\sum_{\gamma \in \mathcal{A}^-} \P_{\gamma}(S^+>\lfloor y \rfloor -kd) \times \ \P_{\beta} (S_{\sigma^-}=kd, A_{\sigma^-}=\gamma)\ .
\end{align*}

This further leads to
\begin{align*}
& \underset{m\rightarrow\infty}{\lim}\P_{\alpha}\left(M_{K_m}\leq \frac{\log m}{\theta^*}+x\right)
=\underset{m\rightarrow\infty}{\lim}
\exp\left \{-\sum_{\beta \in \mathcal{A}^-} mz_{\beta} \P_{\beta}\left(S^+> \left \lfloor\frac{\log(m)}{\theta^*}+x\right\rfloor\right)\right \}\\
& \ \ \ \ \ \ \ \ \ \ \ \ \ \ \ \ \ \ \times\exp\left \{
\sum_{k < 0}\sum_{\gamma \in \mathcal{A}^-} \P_{\gamma}(S^+>\lfloor y \rfloor -kd) \cdot\sum_{\beta \in \mathcal{A}^-}z_{\beta}\P_{\beta}\left(S_{\sigma^-}=kd, A_{\sigma^-}=\gamma\right )
\right \}.
\end{align*}

Since $K_{m(n)} \leq n \leq K_{m(n)+1}$ and $m(n) \longrightarrow \infty$ \textit{a.s.}, Lemma \ref{Km} implies that $\displaystyle \frac{n}{m(n)} \longrightarrow A^*$ \textit{a.s.} Moreover, since $M_{K_{m(n)}} \leq M_n \leq M_{K_{m(n)+1}}$, we finally obtain
\begin{align*}
& \underset{n\rightarrow\infty}{\lim}\P_{\alpha}\left(M_n\leq \frac{\log n}{\theta^*}+x\right )
= \underset{n\rightarrow\infty}{\lim}\P_{\alpha}\left(M_{K_{\lfloor n/A^*\rfloor}}\leq \frac{\log n}{\theta^*}+x\right )\\
&=\underset{n\rightarrow\infty}{\lim}
\exp\left \{-\frac{n}{A^*}\sum_{\beta \in \mathcal{A}^-} z_{\beta} \P_{\beta}\left (S^+> \left \lfloor\frac{\log(n)}{\theta^*}+x\right \rfloor\right )\right \}\\
& \times\exp\left \{ 
\sum_{k < 0}\sum_{\gamma \in \mathcal{A}^-} \P_{\gamma}\left(S^+>\left \lfloor \frac{\log(n)}{\theta^*}+x \right\rfloor -kd\right) \cdot\sum_{\beta \in \mathcal{A}^-}z_{\beta}\P_{\beta} (S_{\sigma^-}=kd, A_{\sigma^-}=\gamma)
\right \}.
\end{align*}

It remains to prove the stated expression for $A^*:= \lim_{m\rightarrow +\infty} \frac{K_m}{m} \ a.s.$ in order to finish the proof. 
Recall that $\sigma^- = K_1$. In Lemma \ref{Km} we proved that
$$A^* =  \sum_{\alpha} z_{\alpha}\E_\alpha(\sigma^-).$$ 

Since $(U_m(\theta))_m$ is a martingale (see Lemma \ref{UmMartingale}) and $\sigma^-$ a stopping time, using the optional sampling theorem we get
$\E_{\alpha}\left [U_{\sigma^-}(\theta)\right ]=\E_{\alpha}\left [U_0(\theta)\right ]=1.$
Consequently,
\begin{eqnarray*}
1&=& \E_{\alpha}\left [\exp(\theta\cdot S_{\sigma^-})
\frac{u_{A_{\sigma^-}}(\theta)}{u_{A_{0}}(\theta)}\frac{1}{\rho(\theta)^{\sigma^-}}\right ]\\
&=&\E_{\alpha}\left [\exp(\theta\cdot S_{\sigma^-})
\frac{u_{A_{\sigma^-}}(\theta)}{u_{\alpha}(\theta)}\frac{1}{\rho(\theta)^{\sigma^-}}\right ]\\
&=& \sum_{\beta}\E_{\alpha}\left [\exp(\theta\cdot S_{\sigma^-})
\frac{u_{\beta}(\theta)}{u_{\alpha}(\theta)}\frac{1}{\rho(\theta)^{\sigma^-}}\big|A_{\sigma^-}=\beta\right ]\cdot \P_{\alpha}(A_{\sigma^-}=\beta)\\
&=&\sum_{\beta}\frac{u_{\beta}(\theta)}{u_{\alpha}(\theta)}\E_{\alpha}\left [\frac{\exp(\theta\cdot S_{\sigma^-})}{\rho(\theta)^{\sigma^-}} \big|A_{\sigma^-}=\beta\right ]\cdot q_{\alpha\beta}.
\end{eqnarray*}
We deduce
$$u_{\alpha}(\theta)=\sum_{\beta}\E_{\alpha}\left [\frac{\exp(\theta\cdot S_{\sigma^-})}{\rho(\theta)^{\sigma^-}}\big| A_{\sigma^-}=\beta \right ]\cdot u_{\beta}(\theta)q_{\alpha\beta} .$$

Derivating the above relation leads to
\begin{align*}
&u'_{\alpha}(\theta)=\\
& \sum_{\beta}q_{\alpha\beta} u_{\beta}(\theta)
\E_{\alpha}\left [ \frac{
S_{\sigma^{-}}\exp(\theta\cdot S_{\sigma^-})\rho(\theta)^{\sigma^-}
-\exp(\theta\cdot S_{\sigma^-})\sigma^-\rho(\theta)^{\sigma^--1}\rho'(\theta)}{\rho(\theta)^{2\sigma^-}} \big| A_{\sigma^-}=\beta\right ] \\
& \hfill \ \  + \sum_{\beta}q_{\alpha\beta} u'_{\beta}(\theta)\E_{\alpha}\left [\frac{\exp(\theta\cdot S_{\sigma^-})}{\rho(\theta)^{\sigma^-}}\big| A_{\sigma^-}=\beta \right ].
\end{align*}
Since $\rho(0)=1$, we obtain for $\theta = 0$:
$$u'_{\alpha}(0)=\sum_{\beta}q_{\alpha\beta}u_{\beta}(0)\left ( \E_{\alpha}\left [ S_{\sigma^-}\big| A_{\sigma^-}=\beta\right ]- \rho'(0)\E_{\alpha}\left [ \sigma^-\big| A_{\sigma^-}=\beta\right ]\right ) + \sum_{\beta}q_{\alpha\beta}u'_{\beta}(0).$$

By the fact that $u(0)=^t(1/r,\ldots,1/r)$, we further get
$$u'_{\alpha}(0)=\frac{1}{r}\E_{\alpha}[S_{\sigma^-}]-
\frac{\rho'(0)}{r} \E_\alpha(\sigma^-)+ \sum_{\beta}q_{\alpha\beta}u'_{\beta}(0).$$
From the last relation we deduce
\begin{equation}\label{eq1}
\sum_{\alpha}z_{\alpha}u'_{\alpha}(0)
=\frac{1}{r}\sum_{\alpha}z_{\alpha}\E_{\alpha}\left [ S_{\sigma^-}\right ]
-\frac{\rho'(0)}{r}\sum_{\alpha} z_{\alpha}\E_\alpha(\sigma^-) +\sum_{\alpha}\sum_{\beta}z_{\alpha} q_{\alpha\beta}u'_{\beta}(0).
\end{equation}

On the other hand, since $z$ is the stationnary frequency vector of the matrix ${\bf Q}$, we have $z=z \cdot {\bf Q}$ 
and thus
\begin{equation}\label{eq:2}
\sum_{\alpha}z_{\alpha}u'_{\alpha}(0)=^tz \cdot u'(0)=^t(z {\bf Q}) \cdot u'(0)=\sum_{\beta} {^t(z {\bf Q})_{\beta}}\cdot u'_{\beta}(0)
=\sum_{\beta}\sum_{\alpha} z_{\alpha} q_{\alpha\beta} u'_{\beta}(0).
\end{equation}
Equations (\ref{eq1}) and (\ref{eq:2}) imply that $\sum_{\alpha}z_{\alpha}\E_{\alpha}\left [ S_{\sigma^-}\right ]=\rho'(0)\cdot\sum_{\alpha} z_{\alpha} \E_\alpha(\sigma^-)$
and thus $A^*=\sum_{\alpha} z_{\alpha} \E_\alpha(\sigma^-)=\frac{1}{\rho'(0)}\sum_{\alpha}z_{\alpha}\E_{\alpha}\left [ S_{\sigma^-}\right ]$. 
Using now the fact that $\rho'(0)=\E[f(A)]$ (see Lemma \ref{lem:rho'}) gives the stated expression for $A^*$. $\hfill\square$


%
%

\section{Applications and computational methods}\label{sec:comp}
 
In order to simplify the presentation, we suppose in this section that $d=1$. Let $-u,\dots,0,\dots,v$ be the possible scores, with $u, v \in \mathbb{N}$.

For $-u \leq j \leq v$, we introduce the matrix ${\bf P^{(j)}}$ with entries 
$$P^{(j)}_{\alpha \beta} := \P_{\alpha}(A_1=\beta, f(A_{1})= j)$$
for $\alpha, \beta \in \mathcal{A}$.
Note that $ P^{(f(\beta))}_{\alpha\beta}=p_{\alpha\beta}$, $P^{(j)}_{\alpha\beta}=0$ if $j\neq f(\beta)$ and ${\bf P}=\sum_{j=-u}^{v}{\bf P^{(j)}}$, where ${\bf P}=(p_{\alpha \beta})_{\alpha, \beta}$ is the transition probability matrix of the Markov chain $(A_i)_i$.

In order to obtain the approximate distribution of $Q_1$ given in Theorem \ref{res:Q1}, we need to compute the quantities $Q^{(\ell)}_{\alpha \beta}$ for $-u \leq \ell \leq v, \alpha, \beta \in \mathcal{A}$ . This is the topic of the next subsection. We denote ${\bf Q^{(\ell)}}$ the matrix $(Q^{(\ell)}_{\alpha \beta})_{\alpha, \beta \in \mathcal{A}}$.

\subsection{Computation of ${\bf Q^{(\ell)}}$ for $-u \leq \ell \leq v$, and  of ${\bf Q}$}\label{SubSec:Q}

Recall that $Q^{(\ell)}_{\alpha \beta}=\P_\alpha(S_{\sigma^-} = \ell, A_{\sigma^-} = \beta)$, and hence $Q^{(\ell)}_{\alpha \beta}=0$ for $\ell\geq 0$ or $\beta \in \mathcal{A} \setminus  \mathcal{A}^-$. Note also that $\sigma^-=1$ if $f(A_1) < 0$. 
Let $-u \leq \ell \leq -1$.
When decomposing with respect to the possible values $j$ of $f(A_{1})$, we obtain:
\begin{align*}
Q^{(\ell)}_{\alpha \beta}&= \P_\alpha( A_{1} = \beta, f(A_1) = \ell) +  \P_\alpha(S_{\sigma^-} = \ell , A_{\sigma^-} = \beta, f(A_1) = 0) \\
& \ \ \ \ + \sum_{j = 1}^{v} \P_\alpha(S_{\sigma^-} = \ell , A_{\sigma^-} = \beta, f(A_1) = j).
\end{align*}
Note that the first term on the right hand side is exactly $P^{(\ell)}_{\alpha \beta}$ defined at the beginning of this section.
We further have, by the law of total probability and the Markov property:
\begin{align*}
  \P_\alpha(S_{\sigma^-} = \ell , A_{\sigma^-} = \beta, f(A_1) = 0) 
&=  \sum_{\gamma} P^{(0)}_{\alpha \gamma} \ \P_\alpha(S_{\sigma^-} = \ell , A_{\sigma^-} = \beta \ | A_1 = \gamma, f(A_1) = 0)  \\
&=  \sum_{\gamma} P^{(0)}_{\alpha \gamma} \ \P_\gamma(S_{\sigma^-} = \ell , A_{\sigma^-} = \beta)=( {\bf P^{(0)}}{\bf Q^{(\ell)}})_{\alpha \beta}.
\end{align*}

Let $j \in \{1,\ldots,v\}$ be fixed. We have
\begin{align*}
\P_\alpha(S_{\sigma^-} &= \ell , A_{\sigma^-} = \beta, f(A_1) = j)= \sum_{\gamma} P^{(j)}_{\alpha \gamma} \ \P_\alpha(S_{\sigma^-} = \ell , A_{\sigma^-} = \beta \ | A_1 = \gamma, f(A_1) = j).
\end{align*}

For every possible $s \geq 1$, we denote $\mathcal{T}_s$ the set of all possible $s$-tuples $t=(t_1,\dots,t_s)$ verifying $-u\leq t_i\leq -1$ for $i=1,\dots, s$, $\ t_1+\dots+t_{s-1}\geq -j >0$ and $t_1+\dots+t_{s}=\ell-j >0$. Decomposing the possible paths from $-k$ to $\ell$ gives 
$$
Q^{(\ell)}_{\alpha \beta}
= P^{(\ell)}_{\alpha \beta}+({\bf P^{(0)}} {\bf Q^{(\ell)}})_{\alpha \beta}+\sum_{j=1}^v\left ( {\bf P^{(j)}} \sum_s \sum_{t\in{\cal T}_s}\prod_{i=1}^s {\bf Q^{(t_i)}}\right)_{\alpha \beta},
$$
hence 
\begin{equation}\label{Eq:RecurrenceQell}
{\bf Q^{(\ell)}}= {\bf P^{(\ell)}}+{\bf P^{(0)}} {\bf Q^{(\ell)}}+\sum_{j=1}^v {\bf P^{(j)}}\sum_s \sum_{t\in{\cal T}_s}\prod_{i=1}^s {\bf Q^{(t_i)}}.
\end{equation}

Recalling that ${\bf Q}=(q_{\alpha\beta})_{\alpha, \beta}$ with $q_{\alpha\beta}=\P_{\alpha}(A_{\sigma^-}=\beta)=\sum_{\ell<0}Q^{(\ell)}_{\alpha \beta}$,
we have
\begin{equation}\label{Eq:Q}
{\bf Q}=\sum_{\ell <0}{\bf Q^{(\ell)}}.
\end{equation}

\noindent \textbf{Example:} In the case where $u=v=1$, we only have the possible values $\ell=-1$, $j=1$, $s=2$ and $t_1=t_2=-1$, thus
\begin{equation}\label{Eq:RecQellExple}
{\bf Q^{(-1)}}={\bf P^{(-1)}}+{\bf P^{(0)}}\cdot {\bf Q^{(-1)}}+{\bf P^{(1)}}({\bf Q^{(-1)}})^2\mbox{ and } {\bf Q}={\bf Q^{(-1)}}.
\end{equation}

\subsection{Computation of $L_{\alpha \beta}^{(\ell)}$ for $0 \leq \ell \leq v$, and of $L_{\alpha}(\infty)$}\label{Subsec:L}
Recall that $L_{\alpha \beta}^{(\ell)}=\P_\alpha(S_{\sigma^+} = \ell ,\sigma^+ < \infty,  A_{\sigma^+} = \beta)$. Denote ${\bf L^{(\ell)}}:=(L^{(\ell)}_{\alpha\beta})_{\alpha,\beta}$. 
First note that $L_{\alpha \beta}^{(\ell)}=0$ for $\ell\leq0$ or $\beta \in \mathcal{A} \setminus \mathcal{A}^+ $. Using a similar method as the one used to obtain $Q^{(\ell)}_{\alpha \beta}$ in the previous subsection, we denote for every possible $s \geq 1$, 
$\mathcal{T}'_s$ the set of all $s$-tuples $t=(t_1,\dots,t_s)$ verifying $1\leq t_i\leq v$ for $i=1,\dots, s$, $\ t_1+\dots+t_{s-1}\leq k$ and $t_1+\dots+t_{s}=\ell+k >0$. 

For every $0<\ell\leq v$ we then have
\begin{equation}\label{Eq:RecurrenceLell}
{\bf L^{(\ell)}}= {\bf P^{(\ell)}}+{\bf P^{(0)}} {\bf L^{(\ell)}}+\sum_{k=1}^u {\bf P^{(-k)}}\sum_s \sum_{t\in{\cal T}'_s}\prod_{i=1}^s {\bf L^{(t_i)}}
\end{equation}

Since $L_{\alpha}(\infty)=\P_{\alpha}(\sigma^+<\infty)=\sum_{\beta}\sum_{\ell =1}^v L_{\alpha \beta}^{(\ell )}$, and denoting by ${\bf L}(\infty)$ the column vector containing all $L_{\alpha}(\infty)$ for $\alpha \in \mathcal{A}$, and by $\1_r$ the column vector of size $r$ with all components equal to 1, we can write
\begin{equation}
{\bf L}(\infty)=\sum_{\ell=1}^v {\bf L^{(\ell)}}\cdot \1_r.
\end{equation}

\noindent \textbf{Example:} In the case where $u=v=1$, equation (\ref{Eq:RecurrenceLell}) gives
\begin{equation}\label{Eq:LCassimple}
{\bf L^{(1)}}={\bf P^{(1)}}+{\bf P^{(0)}}\cdot {\bf L^{(1)}}+{\bf P^{(-1)}}\cdot ({\bf L^{(1)}})^2,
\end{equation}
\begin{equation}\label{Eq:LCassimple2}
{\bf L^{(\ell)}} = 0 \mbox{ for } \ell > 1, \mbox{ thus } {\bf L}(\infty)= {\bf L^{(1)}}\cdot \1_r.
\end{equation}

\subsection{Computation of $F_{S^+,\alpha}{(\ell)}$ for $\ell \geq 0$}\label{Subsec:FS+}
For $\ell\geq 0$ let us denote  ${\bf F}_{S^+,\cdot}{(\ell)}:=(F_{S^+,\alpha}{(\ell)})_{\alpha \in \mathcal{A}}$, seen as a column vector of size $r$. From Theorem \ref{res:exactS+} we deduce that for $\ell=0$ and every $\alpha\in{\cal A}$ we have
$$F_{S^+,\alpha}{(0)}=1 - L_{\alpha}(\infty).$$

For $\ell=1$ and every $\alpha \in \mathcal{A}$ we get
$$F_{S^+,\alpha}{(1)}=1 - L_{\alpha}(\infty) + \sum_{\beta \in \mathcal{A}} L^{(1)}_{\alpha \beta} \ F_{S^+,\beta}{(0)}.$$
With ${\bf L}(\infty)=(L_{\alpha}(\infty))_{\alpha \in \mathcal{A}}$, seen as a column vector, we can write
\begin{align*}
{\bf F}_{S^+,\cdot}{(1)}&=1 -{\bf L}(\infty) +  {\bf L^{(1)}F}_{S^+,\cdot}{(0)},\\
{\bf F}_{S^+,\cdot}{(\ell)}&=1 - {\bf L}(\infty)+\sum_{k=1}^{\ell} {\bf L^{(k)} F}_{S^+,\cdot}{(\ell-k)}, \  \forall\ell\geq 1.
\end{align*}
See Subsection \ref{Subsec:L} for how to compute ${\bf L^{(k)}}$ for $k \geq 1$ and ${\bf L}(\infty)$.

\subsection{Application in a simple case}\label{subsec:simplecase}
Let us consider the simple case where the possible score values are $-1,0,1$, corres\-ponding to the case $u=v=1$.
We will use the results in the previous subsections (see Equations (\ref{Eq:RecQellExple}, \ref{Eq:LCassimple}, \ref{Eq:LCassimple2})) to derive the distribution of the maximal non-negative partial sum $S^+$. This distribution can be determined using the following matrix equalities:

\begin{equation}\label{eq:Linfini}
{\bf L}(\infty)=\left ( \sum_{\beta} L^{(1)}_{\alpha\beta}\right )_{\alpha}={\bf L^{(1)}}\cdot\1_r,
\end{equation}
with ${\bf L^{(1)}}$ given in Equation (\ref{Eq:RecurrenceLell}) and
\begin{eqnarray}\label{Eq:LoiExacteS+CasSimple}
{\bf F}_{S^+,\cdot}{(0)}&=&1-{\bf L}(\infty),\\ 
{\bf F}_{S^+,\cdot}{(\ell)}&=& 1-{\bf L}(\infty)+{\bf L^{(1)} F}_{S^+,\cdot}{(\ell-1)}.						
\end{eqnarray}
This allows to further derive the approximate distributions of $Q_1$ and $M_n$ given in Theorems \ref{res:Q1} and \ref{res:Mn}.

We present hereafter a numerical application for the local score of a DNA sequence. We suppose that we have a Markovian sequence whose possible letters are $\{A,C,G,T\}$ and whose transition probability matrix is given by
$${\bf P}=\left (
\begin{array}{cccc}
1/2&1/6&1/6&1/6\\
1/4&1/4&1/4&1/4\\
1/6&1/6&1/6&1/2\\
1/6&1/6&1/2&1/6\\
\end{array}
\right)\ .$$
We choose the respective scores $-1,-1,0,1$ for the letters $A, C,G,T$ for which Hypothesis (\ref{Hyp:ScoreMoyNeg}) and 
(\ref{Hyp:ProbaScorePos}) are verified. 
We use the successive iteration methodology described in Equation (5.12) of \cite{KDe92}  in order to compute $\bf{L^{(1)}}$ and $\bf{Q^{(-1)}}$, solutions of Equations (\ref{Eq:RecQellExple}) and (\ref{Eq:LCassimple}), from which we derive the formulas proposed in our Theorems \ref{res:exactS+}, \ref{res:Q1} and \ref{res:Mn} for the approximate distributions of $S^+$, $Q_1$ and $M_n$ respectively. We also compute the different approximations proposed in Karlin and Dembo \cite{KDe92}. 
We then compare these results with the  corresponding empirical distributions computed using a Monte Carlo approach based on $10^5$ simulations. We can see in Figure \ref{Fig:Splus}, left panel, that for $n=300$ the empirical  \textit{cdf} of $S^+$ and the one obtained using Theorem \ref{res:exactS+} match perfectly. We can also visualize the fact that Theorem \ref{res:exactS+} improves the approximation of Karlin and Dembo in Lemma 4.3 of \cite{KDe92} for the distribution of $S^+$. The right panel of Figure \ref{Fig:Splus} allows to compare, for different values of the sequence length $n$, the empirical  \textit{cdf} of $S^+$ and the exact \textit{cdf} given in Theorem \ref{res:exactS+}: we can see that our formula performs very satisfactory even for sequence length $n=100$.

In this simple example the approximation of the distribution of $Q_1$ given in Theorem \ref{res:Q1} and the one given in Lemma 4.4 of \cite{KDe92} give quite similar numerical values. 

In Figure \ref{Fig:Mn} we compare three approximations for the \textit{cdf} of $M_n$: the Karlin and Dembo's approximation given in Equation (1.27) of \cite{KDe92} (see also Equation (\ref{Res:AppxMnKDe})), our approximation proposed in Theorem \ref{res:Mn}, and a Monte Carlo approximation. For the simple scoring scheme of this application, the parameter $K^*$ of the Karlin and Dembo's approximation for $M_n$ is given by Equation (5.6) of \cite{KDe92}
$$K^*=(e^{-\theta^*}-e^{-2\theta^*})\cdot\E[-f(A)]\cdot\sum_{\gamma}z_{\gamma}u_{\gamma}(\theta^*)\cdot\sum_{\gamma}w_{\gamma}/u_{\gamma}(\theta ^*).$$
More precisely, in the left panel we plot the probability $p(n,x):=\P\left(M_n\leq \frac{\log(n)}{\theta^*}+x\right)$ as a function of $n$, for a fixed value $x=-8$. This illustrates the asymptotic behavior of this probability with growing $n$. We can also observe the fact that Karlin and Dembo's approximation does not depend on $n$.  
In Figure \ref{Fig:Mn}, right panel, we compare the approximation of Karlin and Dembo \cite{KDe92} for the same probability $p(n,x)$ with our approximation, for varying $x$ and fixed $n=100$. 
We observe that the improvement brought by our approximation is more significant for negative values of $x$. For fixed $n$ and extreme deviations (large $x$) the two approximations are quite similar and accurate.

\begin{flushleft}
\begin{figure}
\centerline{
\begin{tabular}{cc}
\includegraphics[width=0.5\textwidth,trim=0 0 0 50,clip]{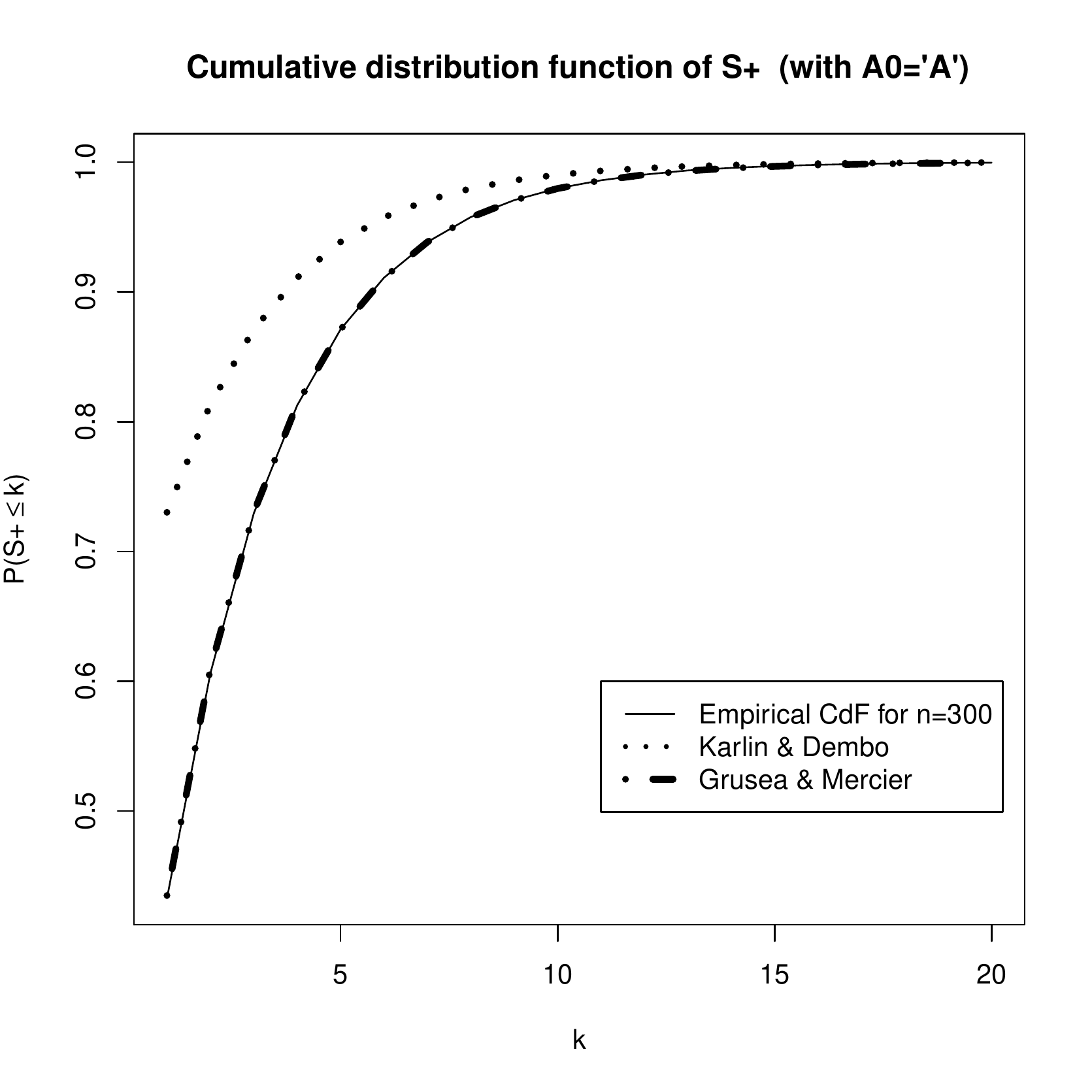}&
\includegraphics[width=0.5\textwidth,trim=0 0 0 50,clip]{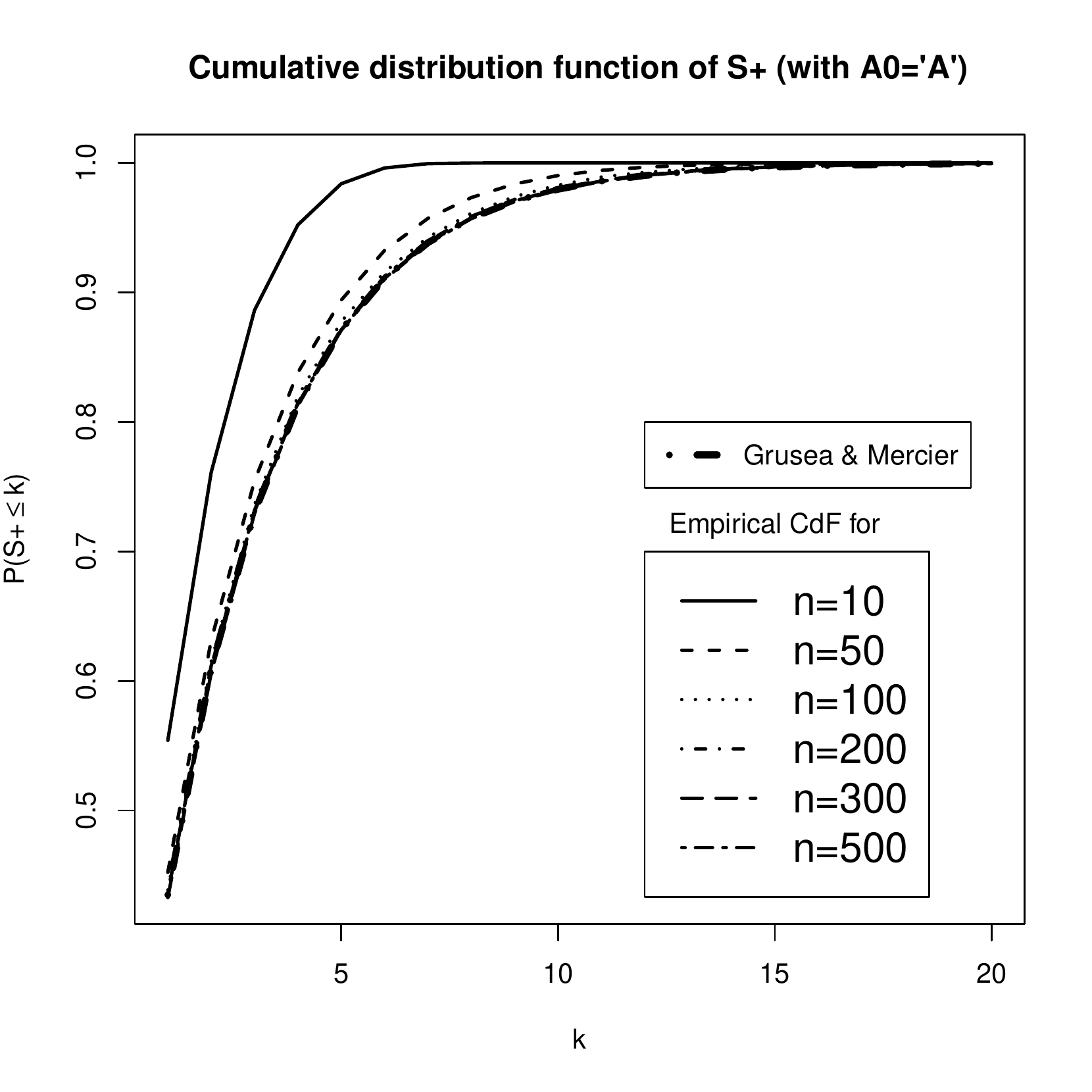}
\end{tabular}
}
\caption{Cumulative distribution function of $S^+$ for the simple scoring scheme $(-1,0,+1)$ and $A_0=$``$A$''. Left panel: Comparison between the approximation of Karlin and Dembo proposed in \cite{KDe92}, a Monte Carlo estimation with sequences of length $n=300$, and our exact formula proposed in Theorem \ref{res:exactS+}. Right panel: Comparison, for different values of $n$, of the Monte Carlo empirical cumulative distribution function and the exact one given in Theorem \ref{res:exactS+}.}
\label{Fig:Splus}
\end{figure}
\end{flushleft}

\begin{flushleft}
\begin{figure}
\centerline{
\begin{tabular}{cc}
\includegraphics[width=0.5\textwidth,trim=0 0 0 50,clip]{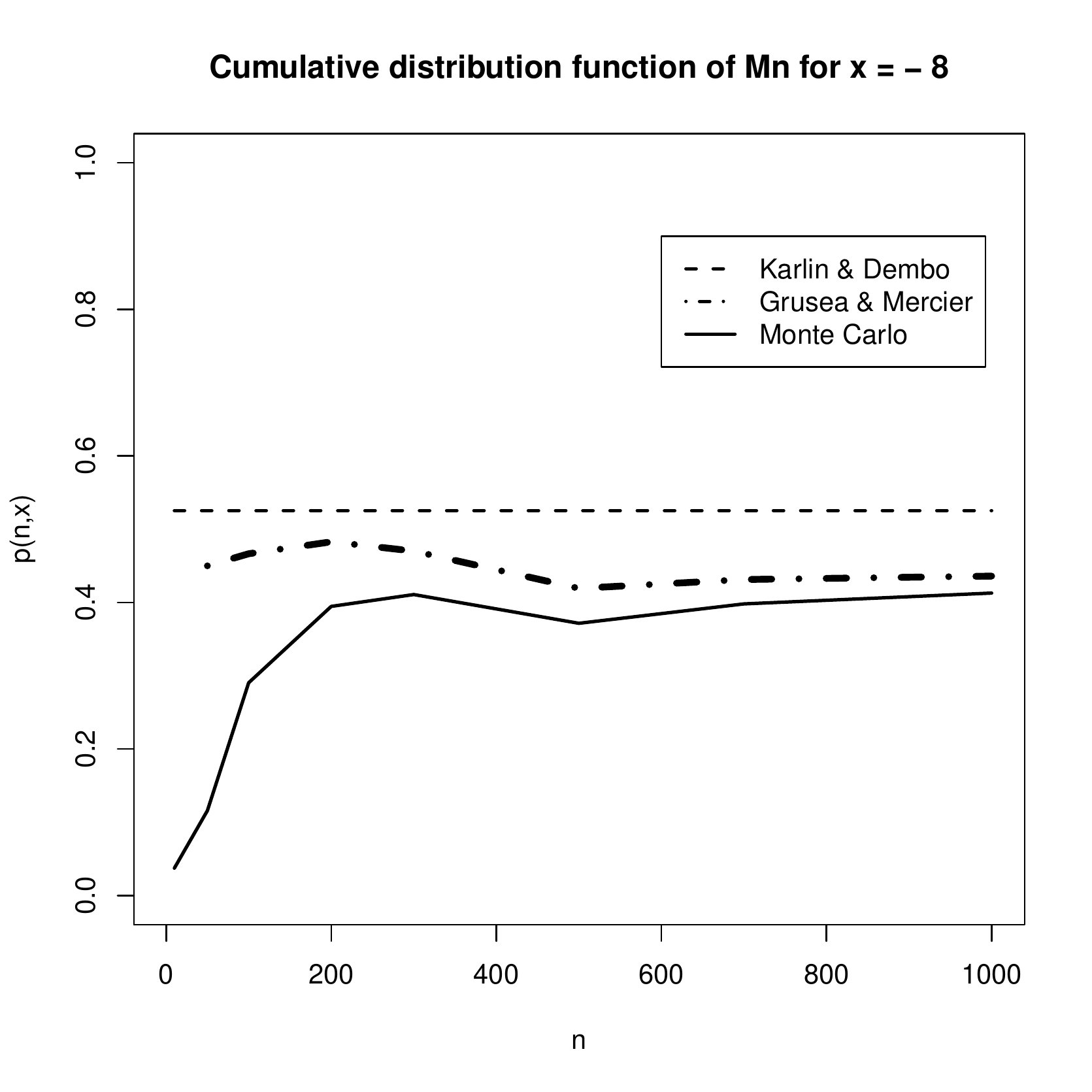}&
\includegraphics[width=0.5\textwidth,trim=0 0 0 50,clip]{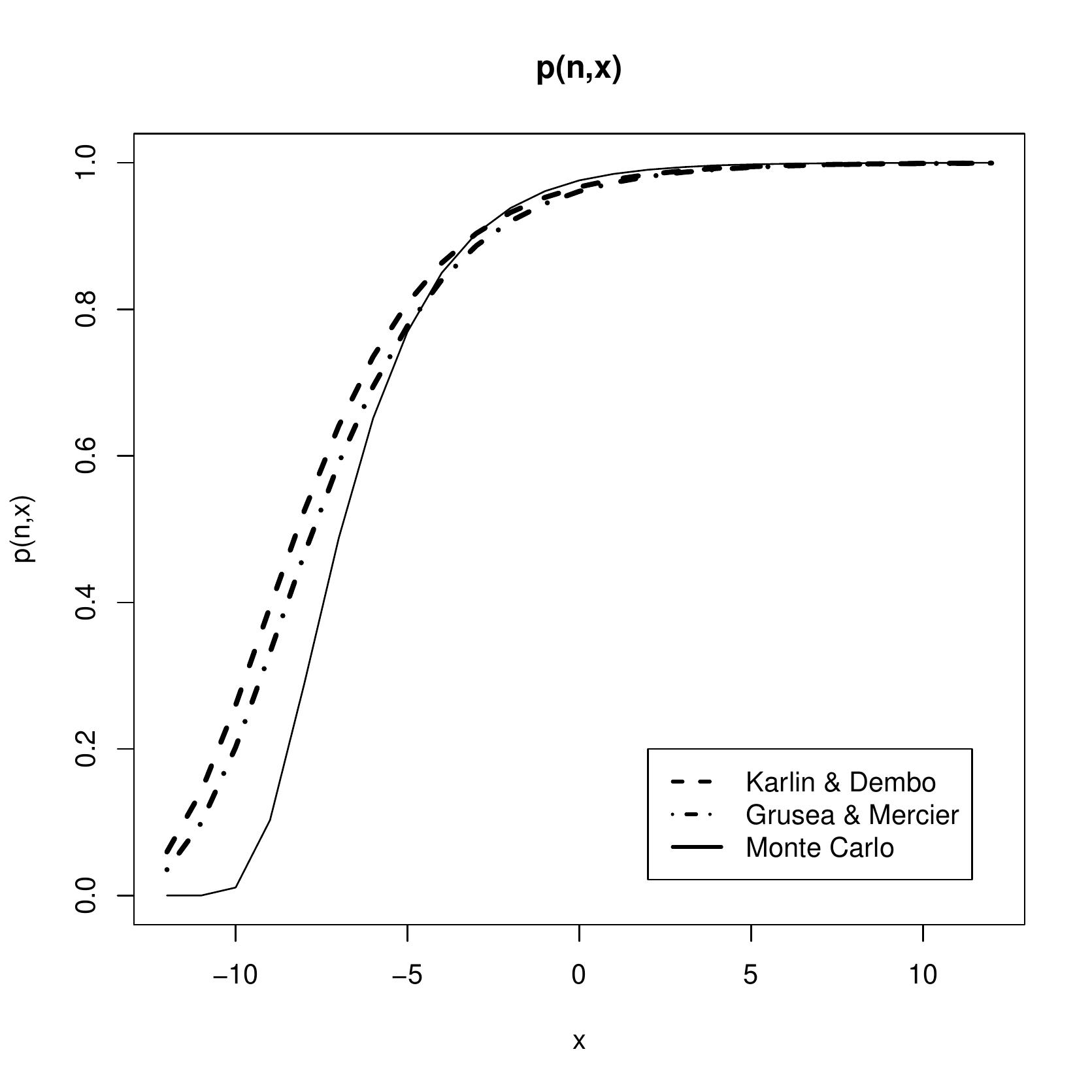}
\end{tabular}
}
\caption{Comparison of the different approximations for $p(n,x)=\P\left(M_n\leq \frac{\log(n)}{\theta^*}+x\right)$ with the simple scoring scheme $(-1,0,+1)$: Karlin and Dembo's result \cite{KDe92} (see Equation (\ref{Res:AppxMnKDe})), our approximation proposed in Theorem \ref{res:Mn} and Monte Carlo estimation. Left panel: $p(n,x)$ as a function of $n$, for fixed $x=-8$. Right panel: $p(n,x)$ as a function of $x$, for fixed $n=100$.}
\label{Fig:Mn}
\end{figure}
\end{flushleft}


%
%
%
%


\vspace{-2.2cm}

\begin{thebibliography}{3}
\footnotesize

\bibitem{AMu75}      
{\sc Athreya, K. B. and Rama Murthy, K.} (1976).
Feller's renewal theorem for systems of renewal equations. 
\emph{J. Indian Inst. Sci.}
{\bf 58(10),} {437--459}.


\bibitem{MCC03}
{\sc Cellier D., Charlot, F. and Mercier, S.} (2003). 
An improved approximation for assessing the statistical significance of molecular sequence features. 
\emph{\JAP},
{\bf 40,} 427--441.

\bibitem{DKa91b}
{\sc Dembo, A. and Karlin, S.} (1991).
\newblock Strong limit theorems of empirical distributions for large segmental exceedances of partial sums of markov variables. \emph{Ann. Probab.}, {\bf 19(4),} 1756--1767. 

\bibitem{DKM98}
Durbin, R. and Eddy, S. and Krogh, A. and Mitchion, G. (1998).
\emph{Biological sequence analysis: Probabilistic Models of Proteins and Nucleic Acids},
Cambridge University Press.

\bibitem{FBM17}
{\sc Fariello  M.-I. and Boitard S. and  Mercier S. and Robelin  D. and Faraut T. and Arnould C. and Le Bihan-Duval E. and Recoquillay J. and Salin G. and Dahais P. and Pitel F. and Leterrier C. and Sancristobal M.} (2017).
A new local score based method applied to behavior-divergent quail lines sequenced in pools precisely detects selection signatures on genes related to autism.
\emph{Molecular Ecology} {\bf 26(14),} 3700--3714.

\bibitem{GRH06}
{\sc Guedj, M. and Robelin,D. and Hoebeke, M. and Lamarine, M. and Wojcik, J. and Nuel, G.} (2006).
Detecting local high-scoring segments: a first-stage approach for genome-wide association studies,\emph{Stat. Appl. Genet. Mol. Biol.},
{\bf 5(1)}.
				
\bibitem{HMe07}
{\sc Hassenforder, C. and Mercier, S.} (2007).
Exact Distribution of the Local Score for Markovian Sequences.
\emph{Ann. Inst. Stat. Math.},
{\bf 59(4),} 741--755.
       

\bibitem{KAl90}
{\sc Karlin, S. and Altschul, S.-F.} (1990).
\newblock Methods for assessing the statistical significance of molecular
  sequence features by using general scoring schemes.
\emph{\PNAS}, {\bf 87,} 2264--2268.

\bibitem{KDe92}
{\sc Karlin, S. and Dembo, A.} (1992).
Limit distributions of maximal segmental score among {M}arkov-dependent partial sums.
\emph{\AAP}, {\bf 24,} 113--140.

\bibitem{KOs87}
{\sc Karlin, S. and Ost, F.} (1987).
Counts of long aligned word matches among random letter sequences.
\emph{\AAP}, {\bf 19,} 293--351.


\bibitem{Lancaster}
{\sc Lancaster, P.} (1969). \emph{Theory of Matrices}, Academic Press, New York.



\bibitem{MDa01}
{\sc Mercier, S. and Daudin, J.J.}  (2001).
Exact distribution for the local score of one i.i.d.\ random sequence.
\emph{J. Comp. Biol.}, {\bf 8(4),} 373--380.


\end{thebibliography}
\end{document}